\documentclass[11pt]{article}
\usepackage[letterpaper,hmargin=1in,vmargin=1.25in]{geometry}

\usepackage{url}
\usepackage{hyperref}
\usepackage{comment}

\usepackage{amsmath}
\usepackage{amssymb}

\usepackage{amsthm}

\theoremstyle{plain}
\newtheorem{theorem}{Theorem}
\newtheorem{lemma}[theorem]{Lemma}
\newtheorem{corollary}[theorem]{Corollary}

\newtheorem{definition}[theorem]{Definition}
\newtheorem{remark}[theorem]{Remark}

\theoremstyle{definition}

\newcommand{\abs}[1]{\left\lvert#1\right\rvert}

\newcommand{\rest}[2]{#1\!\!\restriction_{#2}}

\newcommand{\osg}[1]{\left[#1\right]^{\prec}}

\newcommand{\N}{\mathbb{N}}

\newcommand{\Q}{\mathbb{Q}}
\newcommand{\R}{\mathbb{R}}

\newcommand{\X}{\{0,1\}^*}

\newcommand{\Bm}[2]{\lambda_{#1}\left(#2\right)}

\newcommand{\PS}{\mathbb{P}}%

\newcommand{\noi}{\noindent}

\begin{document}

\begin{center}
{\Large \textbf{An effectivization of the law of large numbers for
algorithmically random sequences
and its absolute speed limit of convergence}}
\end{center}

\vspace{0mm}

\begin{center}
Kohtaro Tadaki
\end{center}

\vspace{-5mm}

\begin{center}
Department of Computer Science, College of Engineering, Chubu University\\
1200 Matsumoto-cho, Kasugai-shi, Aichi 487-8501, Japan\\
E-mail: \textsf{tadaki@isc.chubu.ac.jp}\\
\url{http://www2.odn.ne.jp/tadaki/}
\end{center}

\vspace{-2mm}

\begin{quotation}
\noi\textbf{Abstract.}
The law of large numbers is one of the fundamental properties which
algorithmically random infinite sequences ought to satisfy.
In this paper, we show that the law of large numbers can be effectivized
for an arbitrary Schnorr random infinite sequence,
with respect to an arbitrary computable Bernoulli measure.
Moreover, we show that an absolute speed limit of convergence exists in this effectivization,
and it equals $2$ in a certain sense.
In the paper, we also provide the corresponding effectivization of
almost sure convergence in
the strong law of large numbers, and its absolute speed limit
of convergence,
in the context of probability theory,
with respect to
a large class of probability spaces and
i.i.d.~random variables on them,
which are not necessarily computable.
\end{quotation}

\begin{quotation}
\noi\textit{Key words\/}:
the law of large numbers,
effectivization,
algorithmic randomness,
Schnorr randomness,
Martin-L\"of randomness,
central limit theorem,
Bernoulli measure
\end{quotation}

\section{Introduction}

Algorithmic randomness is a field of mathematics which
enables us to consider the randomness of an individual infinite sequence
\cite{C87b,C02,LV08,N09,DH10}.
Thus,
algorithmically random infinite sequences are the subject of algorithmic randomness.
One of the fundamental properties which algorithmically random infinite sequences ought to satisfy is the \emph{law of large numbers}~%
\cite{vM57,Wa36,Wa37,Ch40,vM64,M66,Sch71}.
In this paper, we study an \emph{effectivization} of the law of large numbers for algorithmically random infinite sequences.

 In algorithmic randomness,
two of the major
and
historically oldest
randomness notions
are
Martin-L\"of randomness~\cite{M66} and its generalization,
Schnorr randomness~\cite{Sch71}.
The law of large numbers
actually
holds
for an arbitrary Martin-L\"of random infinite sequence and, more generally,
for an arbitrary Schnorr random infinite
sequence
\cite{M66,Sch71,C87b,C02,LV08,N09,DH10}.
This result is with respect to Lebesgue measure.
More generally, it is well-known that,
with respect to
a
\emph{computable} Bernoulli measure,
the law of large numbers holds
for an arbitrary Martin-L\"of random
sequence and
further
for an arbitrary Schnorr random
sequence
(see Lutz~\cite{Lutz03}, Downey, Merkle, and Reimann~\cite{DMR06}, and
Downey and Hirschfeldt~\cite{DH10}).

In contrast,
in our former work~\cite{T14,T15},
we showed that,
with respect to an \emph{arbitrary} Bernoulli measure,
the law of large numbers still holds for an arbitrary Martin-L\"of random infinite sequence.
In this result,
the underlying Bernoulli measure is \emph{quite arbitrary}, and therefore
is not required to be computable at all, in particular.
Thus, the computability of the underlying Bernoulli measure is not essential
for a Martin-L\"of random sequence with respect to it to satisfy the law of large numbers.

In this regard, the following question may arise naturally.
\begin{quote}
\textbf{Question}:
What role does
the computability of the underlying Bernoulli measure
play in the law of large numbers for an algorithmically random infinite sequence?
\end{quote}
As an answer to this question,
in our subsequent work~\cite{T21arXiv}, we
showed that
the computability of the underlying Bernoulli measure
leads to an \emph{effectivization} of the law of large numbers.
Actually,
in Theorem~51 of that work, we
showed that
the law of large numbers can be effctivized
for an \emph{arbitrary} Martin-L\"of infinite sequence,
with respect to an \emph{arbitrary computable} Bernoulli measure
(see Theorem~\ref{ML-EffectiveLLN0} below in this paper).

This paper is a sequel to our work~\cite{T21arXiv}, in particular, a sequel to Section~9 of
the work~\cite{T21arXiv}.
In this paper, we show the following:
Frist,
in Theorem~\ref{EffectiveLLN0} below,
we show that
the law of large numbers can be effectivized
for an \emph{arbitrary Schnorr random} infinite sequence,
with respect to an \emph{arbitrary computable} Bernoulli measure.
Thus,
in this paper we generalize over the notion of Schnorr randomness
Theorem~51 of the work~\cite{T21arXiv},
which describes the
original
result regarding the effective convergence of the law of large numbers for
the notion of
Martin-L\"of randomness
in the work~\cite{T21arXiv}.

In this paper, we then investigate a ``converse'' of Theorem~\ref{EffectiveLLN0}.
Namely, we show
in Theorem~\ref{thm:main} below that
(i)
an \emph{absolute speed limit of convergence}
exists in the effectivization of the law of large numbers
which is provided by Theorem~\ref{EffectiveLLN0},
and
(ii)
this speed limit equals $2$, in the sense
stated in
Theorem~\ref{thm:main}.
The central limit theorem plays a crucial role in proving
the converse theorem, Theorem~\ref{thm:main}.

Theorem~\ref{EffectiveLLN0} and Theorem~\ref{thm:main} deal with
an effectivization of the law of large numbers,
and its absolute speed limit of convergence,
in terms of
\emph{relative frequency} of each symbol which occurs in
an arbitrary Schnorr random infinite sequence
with respect to an arbitrary computable Bernoulli measure.
By
Theorems~\ref{EffectiveLLN-RV}, \ref{thm:main-ML}, and \ref{thm:main-Schnorr} below,
we
provide
a similar effectivization of the law of large numbers,
and its absolute speed limit of convergence,
for a \emph{real random variable}
whose underlying process
is described as
an arbitrary Schnorr (or Martin-L\"of) random infinite sequence
with respect to an arbitrary computable Bernoulli measure.

Note that an infinite sequence is
Schnorr random
\emph{almost surely}, with respect to an \emph{arbitrary} Bernoulli measure, i.e.,
the set of all
Schnorr random
infinite sequeces with respect to an \emph{arbitrary} Bernoulli measure
has \emph{probability one}
with respect to that Bernoulli measure (see Theorem~\ref{Schnorr-P-AE} below).
The same holds for Martin-L\"of randomness
with respect to an arbitrary Bernoulli measure.
Thus,
we can prove that
the \emph{property}
on a Schnorr random infinite sequence
or a Martin-L\"of random infinite sequence,
which is
considered
in each of
Theorems~\ref{EffectiveLLN-RV}, \ref{thm:main-ML}, and \ref{thm:main-Schnorr},
holds \emph{almost surely} for any infinite sequence.
In the resulting statements proved in this manner,
which correspond to
Theorems~\ref{EffectiveLLN-RV}, \ref{thm:main-ML}, and \ref{thm:main-Schnorr},
we can eliminate the requirement of the computability of the underling Bernoulli measure,
by using in their proofs
the notions of Schnorr randomness or Martin-L\"of randomness
\emph{with oracle} which computes the underlying Bernoulli measure.
However, we can easily prove
far stronger
results
within the framework of the traditional probability theory,
instead of resorting to the use of algorithmic randomness like the above argument.
Thus,
within the traditional and genuine framework of probability theory,
in this paper
we provide an effectivization of \emph{almost sure convergence} in
the strong law of large numbers, and its absolute speed limit
of convergence,
in a \emph{still more general setting} regarding
a type of the underlying probability
space
and random variables on it
than that
assumed
in Theorems~\ref{EffectiveLLN-RV}, \ref{thm:main-ML}, and \ref{thm:main-Schnorr}.
These results
in probability theory
are given as
Theorems~\ref{th:SLLN-effective-convergence}, \ref{SpeedLimitTheorem-SLLN},
and \ref{thm:main-SLLN} below.
Note that, certainly,
the underlying probability space or random variables on it are not required to be computable at all in any sense
in Theorems~\ref{th:SLLN-effective-convergence}, \ref{SpeedLimitTheorem-SLLN}, and \ref{thm:main-SLLN}.

\subsection{Organization of the paper}

The paper is organized as follows.
We begin in Section~\ref{preliminaries} with some preliminaries to
algorithmic randomness and measure theory.
In Section~\ref{Section:Main-effectiveLLN},
we investigate an effectivization of the law of large numbers,
and its absolute speed limit of convergence,
in terms of
relative frequency of each symbol which occurs in
an arbitrary Schnorr random infinite sequence
with respect to an arbitrary computable Bernoulli measure.
Subsequently,
in Section~\ref{Section:EffectiveLLN=RV}
we investigate a similar effectivization of the law of large numbers,
and its absolute speed limit of convergence,
for a real random variable
whose underlying process
is described as
an arbitrary Schnorr random infinite sequence
with respect to an arbitrary computable Bernoulli measure,
based partially on the results of Section~\ref{Section:Main-effectiveLLN}.
In Section~\ref{Section:EffectiveSLLN},
within the framework of probability theory
we provide
an
effectivization of
almost sure convergence in
the strong law of large numbers, and its absolute speed limit
of convergence,
with respect to
a large class of probability spaces and
independent and identically distributed random variables on them,
which are not necessarily computable.
Note that Section~\ref{Section:EffectiveSLLN}
does not depend on
any results of Section~\ref{preliminaries}, Section~\ref{Section:Main-effectiveLLN}, or Section~\ref{Section:EffectiveLLN=RV}.
Thus, it can be read independently of
these preceding sections.

\section{Preliminaries}
\label{preliminaries}

\subsection{Basic notation and definitions}
\label{basic notation}

We start with some notation about numbers and strings which will be used in this paper.
$\#S$ is the cardinality of $S$ for any set $S$.
$\N=\left\{0,1,2,3,\dotsc\right\}$ is the set of \emph{natural numbers},
and $\N^+$ is the set of \emph{positive integers}.
$\Q$ is the set of \emph{rationals}, and $\R$ is the set of \emph{reals}.
For any $a\in\R$, as usual, $\lceil a\rceil$ denotes the smallest integer greater than or equal to $a$.
A real $a\in\R$ is called \emph{computable} if
there exists a total recursive function $f\colon\N\to\Q$ such that
$\abs{a-f(k)} < 2^{-k}$ for all $k\in\N$.

An \emph{alphabet} is
a non-empty
finite set.
Let $\Omega$ be an arbitrary alphabet throughout the rest of this
section.
A \emph{finite string over $\Omega$} is a finite sequence of elements from the alphabet $\Omega$.
We use $\Omega^*$ to denote the set of all finite strings over $\Omega$,
which contains the \emph{empty string} denoted by $\lambda$.
We use $\Omega^+$ to denote the set $\Omega^*\setminus\{\lambda\}$.
For any $\sigma\in\Omega^*$, $\abs{\sigma}$ is the \emph{length} of $\sigma$.
Therefore $\abs{\lambda}=0$.
For any $\sigma\in\Omega^+$ and $k\in\N^+$ with $k\le\abs{\sigma}$,
we use $\sigma(k)$ to denote the $k$th element in $\sigma$.
Therefore, we have $\sigma=\sigma(1)\sigma(2)\dots\sigma(\abs{\sigma})$
for every $\sigma\in\Omega^+$.
For any $n\in\N$, we use $\Omega^n$ to denote the set~$\{\,x\mid x\in\Omega^*\;\&\;\abs{x}=n\}$.
A subset $S$ of $\Omega^*$ is called
\emph{prefix-free}
if no string in $S$ is a prefix of another string in $S$.

An \emph{infinite sequence over $\Omega$} is an infinite sequence of elements from the alphabet $\Omega$,
where the sequence is infinite to the right but finite to the left.
We use $\Omega^\infty$ to denote the set of all infinite sequences over $\Omega$.

Let $\alpha\in\Omega^\infty$.
For any $n\in\N$
we denote by $\rest{\alpha}{n}\in\Omega^*$ the first $n$ elements
in the infinite sequence $\alpha$,
and
for any $n\in\N^+$ we denote
by $\alpha(n)$ the $n$th element in $\alpha$.
Thus, for example, $\rest{\alpha}{4}=\alpha(1)\alpha(2)\alpha(3)\alpha(4)$, and $\rest{\alpha}{0}=\lambda$.

For any $S\subset\Omega^*$, the set
$\{\alpha\in\Omega^\infty\mid\exists\,n\in\N\;\rest{\alpha}{n}\in S\}$
is denoted by $\osg{S}$.
For any $\sigma\in\Omega^*$, we denote by $\osg{\sigma}$ the set $\osg{\{\sigma\}}$, i.e.,
the set of all infinite sequences over $\Omega$ extending $\sigma$.
Therefore $\osg{\lambda}=\Omega^\infty$.
A subset $\mathcal{R}$ of $\Omega^\infty$ is \emph{open} if
$\mathcal{R}=\osg{S}$ for some $S\subset\Omega^*$.
A class $\mathcal{F}$ of subsets of $\Omega^\infty$ is called
a \emph{$\sigma$-field in $\Omega^\infty$}
if  $\mathcal{F}$ includes $\Omega^\infty$, is closed under complements,
and is closed under the formation of countable unions.
The \emph{Borel class} $\mathcal{B}_{\Omega}$ is the $\sigma$-field \emph{generated by}
all open sets on $\Omega^\infty$.
Namely, the Borel class $\mathcal{B}_{\Omega}$ is defined
as the intersection of all the $\sigma$-fields on $\Omega^\infty$ containing
all open sets on $\Omega^\infty$.
The pair $(\Omega^\infty,\mathcal{B}_{\Omega})$ forms a \emph{measurable space}.
See Billingsley~\cite{B95}, Chung~\cite{Chung01},
Nies~\cite[Sections 1.8 and 1.9]{N09},
Durrett~\cite{Durrett19}, and
Klenke~\cite{Klenke20} for measure theory and probability theory.

We write ``r.e.'' instead of ``recursively enumerable.''

\subsection{Finite probability spaces and Bernoulli measure}
\label{FPS-BM}

Let $\Omega$ be an alphabet.
A \emph{finite probability space on $\Omega$} is a function $P\colon\Omega\to\R$
such that
(i) $P(a)\ge 0$ for every $a\in \Omega$, and
(ii) $\sum_{a\in \Omega}P(a)=1$.
The set of all finite probability spaces on $\Omega$ is denoted by $\PS(\Omega)$.
A finite probability space $P\in\PS(\Omega)$ is called \emph{computable} if
$P(a)$ is a computable real for every $a\in\Omega$.

Let $P\in\PS(\Omega)$.
For each $\sigma\in\Omega^*$, we use $P(\sigma)$ to denote
$P(\sigma_1)P(\sigma_2)\dots P(\sigma_n)$
where $\sigma=\sigma_1\sigma_2\dots\sigma_n$ with $\sigma_i\in\Omega$.
Therefore $P(\lambda)=1$, in particular.
For each subset $S$ of $\Omega^*$, we use $P(S)$ to denote
\[
  \sum_{\sigma\in S}P(\sigma).
\]
Therefore $P(\emptyset)=0$, in particular.
The \emph{Bernoulli measure} $\lambda_{P}$ is a probability measure on
the measurable space $(\Omega^\infty,\mathcal{B}_{\Omega})$ such that
\begin{equation}\label{pBm}
  \Bm{P}{\osg{\sigma}}=P(\sigma)
\end{equation}
for every $\sigma\in \Omega^*$.
The Bernoulli measure $\lambda_{P}$
is called \emph{computable} if
there exists a total recursive function $f\colon\Omega^*\times\N\to\Q$ such that
$\abs{\Bm{P}{\osg{\sigma}}-f(\sigma,k)} < 2^{-k}$
for all $\sigma\in\Omega^*$ and $k\in\N$.
It is easy to see that $P$ is computable
if and only if $\lambda_{P}$ is computable.

\subsection{\boldmath Martin-L\"of $P$-randomness and Schnorr $P$-randomness}
\label{ML-Schnorr-P-R}

Martin-L\"of randomness and Schnorr randomness are
two of major randomness notions in algorithmic randomness.
In the first half of this paper, we investigate
an effectivization of the law of large numbers, regarding
Martin-L\"of randomness and Schnorr randomness with respect to a Bernoulli measure.

Martin-L\"of randomness with respect to
a
Bernoulli measure,
which is called \emph{Martin-L\"of $P$-randomness} in this paper, is defined as follows:

\begin{definition}[Martin-L\"of $P$-randomness, Martin-L\"{o}f \cite{M66}]
\label{ML_P-randomness}
Let $\Omega$ be an alphabet, and let $P\in\PS(\Omega)$.
\begin{enumerate}
  \item A subset $\mathcal{C}$ of $\N^+\times \Omega^*$ is called a \emph{Martin-L\"{o}f $P$-test} if
    $\mathcal{C}$ is an r.e.~set
    and
    for every $n\in\N^+$ it holds that
    \[\Bm{P}{\osg{\mathcal{C}_n}}<2^{-n},\]
    where $\mathcal{C}_n$ denotes the set
    $\left\{\,
      \sigma\mid (n,\sigma)\in\mathcal{C}
    \,\right\}$.
  \item For any $\alpha\in\Omega^\infty$ and Martin-L\"{o}f $P$-test $\mathcal{C}$,
    we say that $\alpha$ \emph{passes} $\mathcal{C}$ if there exists $n\in\N^+$ such that
    $\alpha\notin\osg{\mathcal{C}_n}$.
  \item For any $\alpha\in\Omega^\infty$, we say that $\alpha$ is \emph{Martin-L\"{o}f $P$-random} if
    for every Martin-L\"{o}f $P$-test $\mathcal{C}$
    it holds that $\alpha$ passes $\mathcal{C}$.\qed
\end{enumerate}
\end{definition}

In the case where $\Omega=\{0,1\}$ and $P$ satisfies that $P(0)=P(1)=1/2$,
the Martin-L\"of $P$-randomness results in
the usual Martin-L\"of randomness
for an infinite binary sequence with respect to Lebesgue measure.
The notion of Martin-L\"of $P$-randomness was
introduced by Martin-L\"{o}f~\cite{M66}, as well as the notion of
the usual
Martin-L\"of randomness.

In the field of algorithmic randomness, the notion of Schnorr randomness~\cite{Sch71} is
one of major randomness notions
strictly
weaker than
the notion of Martin-L\"of randomness~\cite{M66},
with respect to Lebesgue measure
(see e.g.~Nies~\cite{N09} and Downey and Hirschfeldt~\cite{DH10} for the detail of
the relation among these randomness notions).
The notion of Schnorr randomness
is naturally generalized over
a
Bernoulli measure
as follows.

\begin{definition}[Schnorr $P$-randomness, Schnorr~\cite{Sch71}]
\label{Schnorr_P-randomness}
Let $\Omega$ be an alphabet, and let $P\in\PS(\Omega)$.
\begin{enumerate}
  \item A subset $\mathcal{C}$ of $\N^+\times \Omega^*$ is called
    a \emph{Schnorr $P$-test} if $\mathcal{C}$ is a Martin-L\"{o}f $P$-test and moreover
    $\Bm{P}{\osg{\mathcal{C}_n}}$ is uniformly compuatble in $n$, i.e.,
    there exists a total recursive function $f\colon\N^+\times\N\to\Q$ such that
    \[
      \abs{\Bm{P}{\osg{\mathcal{C}_n}}-f(n,k)} < 2^{-k}
    \]
    for all $n\in\N^+$ and $k\in\N$,
    where $\mathcal{C}_n$ denotes the set
    $\left\{\,
      \sigma\mid (n,\sigma)\in\mathcal{C}
    \,\right\}$.
  \item For any $\alpha\in\Omega^\infty$, we say that $\alpha$ is
    \emph{Schnorr $P$-random} if for every Schnorr $P$-test $\mathcal{C}$
    it holds that $\alpha$ passes $\mathcal{C}$.\qed
\end{enumerate}
\end{definition}

In the case where $\Omega=\{0,1\}$ and $P$ satisfies that $P(0)=P(1)=1/2$,
the Schnorr $P$-randomness
results in
the usual Schnorr randomness
for an infinite binary sequence with respect to Lebesgue measure.
The notion of Schnorr $P$-randomness was
introduced by Schnorr~\cite{Sch71}, as well as the notion of
the usual Schnorr randomness.

In Definition~\ref{ML_P-randomness}, we do not require that
each
set $\mathcal{C}_n$ is prefix-free in the definition of a Martin-L\"{o}f $P$-test $\mathcal{C}$.
However,
based on Lemma~\ref{impose-prefix-freeness} below,
we can freely impose this requirement
``$\mathcal{C}_n$ is prefix-free for all $n\in\N^+$''
on an arbitrary Martin-L\"{o}f $P$-test $\mathcal{C}$,
while keeping the set of all infinite sequences over $\Omega$
which
pass
$\mathcal{C}$ the same:

\begin{lemma}\label{impose-prefix-freeness}
Let $\Omega$ be an alphabet.
For every r.e.~subset $\mathcal{C}$ of $\N^+\times \Omega^*$
there exists an r.e.~subset $\mathcal{D}$ of $\N^+\times \Omega^*$ such that
$\mathcal{D}_n$ is a prefix-free subset of $\Omega^*$ and
$\osg{\mathcal{C}_n}=\osg{\mathcal{D}_n}$ for every $n\in\N^+$,
where $\mathcal{C}_n$ and $\mathcal{D}_n$ denote
the sets
$\left\{\,
    \sigma\mid (n,\sigma)\in\mathcal{C}
\,\right\}$
and
$\left\{\,
    \sigma\mid (n,\sigma)\in\mathcal{D}
\,\right\}$,
respectively.
\qed
\end{lemma}

For an explicit proof of Lemma~\ref{impose-prefix-freeness},
see for instance Tadaki~\cite[Lemma~9]{T21arXiv}.

Since there are only countably infinitely many algorithms,
it is easy to show the following theorem.

\begin{theorem}
\label{Schnorr-P-AE}
Let $\Omega$ be an alphabet, and let $P\in\PS(\Omega)$.
Then (i)~$\mathrm{ML}_P\subset \mathrm{S}_P$,
(ii)~$\mathrm{ML}_P,\mathrm{S}_P\in\mathcal{B}_{\Omega}$,
and (iii)~$\Bm{P}{\mathrm{ML}_P}=\Bm{P}{\mathrm{S}_P}=1$,
where $\mathrm{ML}_P$ denotes the set of all Martin-L\"of $P$-random infinite
sequences over $\Omega$,
and $\mathrm{S}_P$ denotes the set of all Schnorr $P$-random infinite
sequences over $\Omega$.
\qed
\end{theorem}

Finally, note that, in Definitions~\ref{ML_P-randomness} and \ref{Schnorr_P-randomness},
and in Theorem~\ref{Schnorr-P-AE},
the finite probability space $P$ is quite arbitrary and thus
$P$ is not required to be computable,
in particular.

\section{\boldmath Effectivization of the law of large numbers for a Schnorr $P$ random sequence, and its speed limit of convergence}
\label{Section:Main-effectiveLLN}

In this section,
we investigate an effectivization of the law of large numbers,
and its absolute speed limit of convergence,
in terms of relative frequency of each symbol which occurs in
an arbitrary Schnorr $P$-random infinite sequence
where $P$ is
an arbitrary
computable finite probability space.

We first review the result of Section~9 of Tadaki~\cite{T21arXiv}.
Let $\Omega$ be an alphabet, and let $P\in\PS(\Omega)$.
For any $\alpha\in\Omega^\infty$,
we say that the law of large numbers holds for the infinite sequence $\alpha$
with respect to the Bernoulli measure $\lambda_{P}$ if
for every real $\varepsilon>0$
there exists
$n_0\in\N^+$
such that
for every $n\in\N^+$ if
$n\ge n_0$
then
for every $a\in\Omega$ it holds that
\begin{equation}\label{LLN-classical}
  \abs{\frac{N_a(\rest{\alpha}{n})}{n}-P(a)}<\varepsilon.
\end{equation}
By proving Theorem~\ref{ML-EffectiveLLN0} below, Tadaki~\cite{T21arXiv} showed that
for every Martin-L\"of $P$-random infinite sequence $\alpha$,
there exists an \emph{effective} procedure which computes the positive integer $n_0$
in the above statement~\eqref{LLN-classical} for any given rational $\varepsilon>0$,
provided that
the underlying Bernoulli measure~$\lambda_{P}$ is \emph{computable}.
This type of convergence is an \emph{effective convergence} in the sense of Pour-El and Richards~\cite{PR89}.

\begin{theorem}
[Tadaki~\cite{T21arXiv}]
\label{ML-EffectiveLLN0}
Let $\Omega$ be an alphabet, and let $P\in\PS(\Omega)$.
Suppose that $P$ is computable.
Let $\alpha\in\Omega^\infty$. Suppose that $\alpha$ is Martin-L\"of $P$-random.
Let $\varepsilon$ be an arbitrary positive real.
Then
there exists $M\in\N^+$ such that for every $n\ge M$ and every $k\in\N^+$ if $k\ge n^{2+\varepsilon}$ then
for every $a\in\Omega$ it holds that
\begin{equation*}\label{ML-EffectiveLLN0-eq}
  \abs{\frac{N_a(\rest{\alpha}{k})}{k}-P(a)}<\frac{1}{n},
\end{equation*}
where $N_a(\sigma)$ denotes the number of the occurrences of $a$ in $\sigma$
for every $a\in\Omega$ and $\sigma\in\Omega^*$.
\qed
\end{theorem}

In this section, we prove Theorem~\ref{EffectiveLLN0} below,
which generalizes Theorem~\ref{ML-EffectiveLLN0} over the notion of Schnorr $P$-randomness.
The poof of Theorem~\ref{EffectiveLLN0} is given in
Section~\ref{Subsection:Proof-of-TheoremEffectiveLLN0} below.

\begin{theorem}
[Effectivization of the law of large numbers for a Schnorr $P$-random sequence I]
\label{EffectiveLLN0}
Let $\Omega$ be an alphabet, and let $P\in\PS(\Omega)$.
Suppose that $P$ is computable.
Let $\alpha\in\Omega^\infty$. Suppose that $\alpha$ is Schnorr $P$-random.
Let $\varepsilon$ be an arbitrary positive real.
Then
there exists $M\in\N^+$ such that for every $n\ge M$ and every $k\in\N^+$ if $k\ge n^{2+\varepsilon}$ then
for every $a\in\Omega$ it holds that
\begin{equation}\label{EffectiveLLN0-eq}
  \abs{\frac{N_a(\rest{\alpha}{k})}{k}-P(a)}<\frac{1}{n},
\end{equation}
where $N_a(\sigma)$ denotes the number of the occurrences of $a$ in $\sigma$
for every $a\in\Omega$ and $\sigma\in\Omega^*$.
\qed
\end{theorem}

Theorem~\ref{EffectiveLLN0} results in
the following theorem.
Theorem~\ref{EffectiveLLN}
describes just an effectivization of the law of large numbers
for an arbitrary Schnorr $P$-random infinite sequence
where $P$ is an arbitrary computable finite probability space.

\begin{theorem}%
[Effectivization of the law of large numbers for a Schnorr $P$-random sequence II]
\label{EffectiveLLN}
Let $\Omega$ be an alphabet, and let $P\in\PS(\Omega)$.
Suppose that $P$ is computable.
Let $\alpha\in\Omega^\infty$. Suppose that $\alpha$ is Schnorr $P$-random.
Let $\epsilon$ be an arbitrary positive rational. 
Then there exists a
primitive recursive
function $f\colon\N\to\N$ which satisfies the following two conditions~(i) and (ii):
\begin{enumerate}
  \item For every $n\in\N^+$ and every $k\in\N^+$
    if $k\ge f(n)$ then
    for every $a\in\Omega$ it holds that
    \begin{equation*}%
      \abs{\frac{N_a(\rest{\alpha}{k})}{k}-P(a)}<\frac{1}{n},
    \end{equation*}
    where $N_a(\sigma)$ denotes the number of the occurrences of $a$ in $\sigma$
    for every $a\in\Omega$ and $\sigma\in\Omega^*$.
\item The
primitive
recursive function $f$ has the following form:
\[f(n)=\left\lceil n^{2+\epsilon} \right\rceil\] for all but finitely many
$n\in\N$.
\end{enumerate}
\end{theorem}

\begin{proof}
We define a function $g\colon\N\to\N$ by $g(n):=\left\lceil n^{2+\epsilon} \right\rceil$.
Then, on the one hand, it is easy to see that
the function $g\colon\N\to\N$ is a primitive recursive function.
On the other hand, it follows from Theorem~\ref{EffectiveLLN0} that
there exists $M\in\N^+$ such that for every $n\ge M$ and every $k\in\N^+$ if $k\ge g(n)$ then
for every $a\in\Omega$ it holds that
\begin{equation}\label{EffectiveLLNpr-eq}
  \abs{\frac{N_a(\rest{\alpha}{k})}{k}-P(a)}<\frac{1}{n}.
\end{equation}
Then we
define a function $f\colon\N\to\N$ by the condition that
$f(n):=g(n)$ if $n\ge M$ and $f(n):=g(M)$ otherwise.
Since $g\colon\N\to\N$ is a primitive recursive function,
it is easy to see that $f\colon\N\to\N$ is still a primitive recursive function.
Then, based on \eqref{EffectiveLLNpr-eq}, it is
easy to check that
the conditions~(i) and (ii) of Theorem~\ref{EffectiveLLN} hold for this $f$.
\end{proof}

Theorem~\ref{EffectiveLLN}
leads to the following corollary, in particular, for the notion of the
usual
Schnorr randomness for an infinite binary sequence with respect to Lebesgue measure.

\begin{corollary}
[Effectivization of the law of large numbers for a Schnorr random sequence]
Let $\alpha$ be an arbitrary Schnorr random infinite binary sequence.
Then there exists a primitive recursive function $f\colon\N\to\N$ such that
for every $n\in\N^+$ and every $k\in\N^+$ if $k\ge f(n)$ then
for every $a\in\{0,1\}$ it holds that
\begin{equation*}
  \abs{\frac{N_a(\rest{\alpha}{k})}{k}-\frac{1}{2}}<\frac{1}{n},
\end{equation*}
where $N_a(\sigma)$ denotes the number of the occurrences of $a$ in $\sigma$
for every $a\in\{0,1\}$ and $\sigma\in\X$.
\end{corollary}

\begin{proof}
Let
$U$
be a finite probability space on $\{0,1\}$ such that $U(0)=U(1)=1/2$.
Then the infinite binary sequence $\alpha$ is Schnorr $U$-random. 
Thus, since $1/2$ is a computable real, the result follows from Theorem~\ref{EffectiveLLN}.
\end{proof}

Theorem~\ref{SpeedLimitTheorem} below gives a converse of Theorem~\ref{EffectiveLLN0}.
The proof of Theorem~\ref{SpeedLimitTheorem} is given in
Section~\ref{Proof-of-TheoremSpeedLimitTheorem} below.
The central limit theorem plays a crucial role in proving
Theorem~\ref{SpeedLimitTheorem}.

\begin{theorem}%
[Convergence speed limit theorem]
\label{SpeedLimitTheorem}
Let $\Omega$ be an alphabet, and let $P\in\PS(\Omega)$.
Suppose that $P$ is computable.
Let $\alpha\in\Omega^\infty$ and let $a\in\Omega$.
Suppose that $0<P(a)<1$ and
there exists $n_0\in\N^+$
such that for every $n\ge n_0$ it holds that
\begin{equation*}%
  \abs{N_a(\rest{\alpha}{4^n})-4^n P(a)}\le 2^n,
\end{equation*}
where $N_a(\sigma)$ denotes the number of the occurrences of $a$ in $\sigma$
for every
$\sigma\in\Omega^*$.
Then $\alpha$ is not Schnorr  $P$-random.
\qed
\end{theorem}

Theorem~\ref{EffectiveLLN0} and Theorem~\ref{SpeedLimitTheorem} together result in
Theorem~\ref{thm:main} below.
Let $\Omega$ be an alphabet.
A finite probability space $P\in\PS(\Omega)$ is called \emph{non-trivial}
if $P(a)<1$ for every $a\in\Omega$,
i.e., if there exists $a\in\Omega$ such that $0<P(a)<1$.

\begin{theorem}
[Main result I regarding algorithmic randomness]
\label{thm:main}
Let $\Omega$ be an alphabet, and let $P\in\PS(\Omega)$.
Suppose that $P$ is a non-trivial computable finite probability space.
Let $\alpha\in\Omega^\infty$. Suppose that $\alpha$ is Schnorr $P$-random.
Let $t$ be an arbitrary positive real. 
Then the following conditions~(i) and (ii) are equivalent to each other:
\begin{enumerate}
\item There exists $M\in\N^+$ such that for every $n\ge M$ and every $k\in\N^+$
  if $k\ge n^t$ then for every $a\in\Omega$ it holds that
  \begin{equation*}
    \abs{\frac{N_a(\rest{\alpha}{k})}{k}-P(a)}<\frac{1}{n},
  \end{equation*}
  where $N_a(\sigma)$ denotes the number of the occurrences of $a$ in $\sigma$
  for every $a\in\Omega$ and $\sigma\in\Omega^*$.
  \item $t>2$.
\end{enumerate}
\end{theorem}

\begin{proof}
On the one hand, the implication~(ii) $\Rightarrow$ (i) of Theorem~\ref{thm:main}
follows immediately from Theorem~\ref{EffectiveLLN0}.
On the other hand,  the implication~(i) $\Rightarrow$ (ii) of Theorem~\ref{thm:main}
is proved as follows:
Suppose that (i) of Theorem~\ref{thm:main} holds.
Let us assume contrarily that $t\le 2$.
Note that there exists $a\in\Omega$ such that $0<P(a)<1$,
since $P$ is a non-trivial.
Then, for this $a$, there exists $M\in\N^+$ such that for every $n\ge M$ it holds that
\begin{equation*}%
  \abs{\frac{N_a(\rest{\alpha}{n^2})}{n^2}-P(a)}\le\frac{1}{n},
\end{equation*}
and therefore there exists
$n_0\in\N^+$
such that
$\abs{N_a(\rest{\alpha}{4^n})-4^n P(a)}\le 2^n$
for all $n\ge n_0$.
Thus
it follows from Theorem~\ref{SpeedLimitTheorem} that
$\alpha$ is not Schnorr $P$-random.
However, this contradicts the assumption of the theorem,
and therefore we must have that $t>2$.
This completes the proof.
\end{proof}

\begin{remark}
Let $\Omega$ be an alphabet, and let $P\in\PS(\Omega)$.
Let $\alpha\in\Omega^\infty$.
Suppose that $\alpha$ is Schnorr $P$-random.
In Theorem~\ref{thm:main},
we further make the assumption that
$P$ is non-trivial.
Contrarily,
suppose that this assumption does not hold.
Then
we have that
$P(a)=0$ or $P(a)=1$ for every $a\in\Omega$.
Therefore, due to (i) of
Theorem~\ref{zero_probability} and Theorem~\ref{one_probability}
in Section~\ref{Subsection:Proof-of-TheoremEffectiveLLN0} below,
the following holds trivially:
For every real $t>0$ it holds that for every $n\in\N^+$ and every $k\in\N^+$ if $k\ge n^{t}$ then for every $a\in\Omega$ it holds that
\begin{equation*}%
  \abs{\frac{N_a(\rest{\alpha}{k})}{k}-P(a)}<\frac{1}{n},
\end{equation*}
where $N_a(\sigma)$ denotes the number of the occurrences of $a$ in $\sigma$
for every $a\in\Omega$ and $\sigma\in\Omega^*$.
Hence,
the implication~(i) $\Rightarrow$ (ii) of Theorem~\ref{thm:main} does not hold
without the assumption of the non-triviality of $P$.
\qed
\end{remark}

\subsection{The proof of Theorem~\ref{EffectiveLLN0}}
\label{Subsection:Proof-of-TheoremEffectiveLLN0}

In this subsection, we prove Theorem~\ref{EffectiveLLN0}.
For that purpose, we
first prove
Theorems~\ref{zero_probability}, \ref{one_probability}, and \ref{contraction} below.
The proofs of
Theorems~\ref{zero_probability}, \ref{one_probability}, and \ref{contraction} are
obtained
by adapting
to Schnorr $P$-randomness 
the proofs of Theorems~13, 12, and 16 of Tadaki~\cite{T21arXiv},
respectively,
which are the corresponding original results
for Martin-L\"of $P$-randomness
in Tadaki~\cite{T21arXiv}.

\begin{theorem}\label{zero_probability}
Let $\Omega$ be an alphabet, and let $P\in\PS(\Omega)$.
\begin{enumerate}
\item
  Let $a\in\Omega$ and $\alpha\in\Omega^\infty$.
  Suppose that $\alpha$ is Schnorr $P$-random and $P(a)=0$.
  Then $\alpha$ does not contain $a$.
\item Actually,
  there exists a single Schnorr $P$-test
  $\mathcal{C}\subset\N^+\times\Omega^*$
  such that, for every $\alpha\in \Omega^\infty$, if $\alpha$ passes
  $\mathcal{C}$
  then $\alpha$ does not contain any element of $P^{-1}(\{0\})$.
\end{enumerate}
\end{theorem}

\begin{proof}
It is sufficient to prove
the result~(ii)
of Theorem~\ref{zero_probability}.
For that purpose, we first define
a subset $S$ of $\Omega^*$
as $\Omega^*\setminus(\Omega\setminus P^{-1}(\{0\}))^*$,
and then define $\mathcal{C}$ as the set $\{(n,\sigma)\mid n\in\N^+\;\&\;\sigma\in S\}$.
Since $S$ is r.e., $\mathcal{C}$ is also r.e., obviously.
Moreover, since $P(\sigma)=0$ for every $\sigma\in S$,
we have
that
$\Bm{P}{\osg{\mathcal{C}_n}}=P(\mathcal{C}_n)=P(S)=0$ for each $n\in\N^+$.
Hence, $\mathcal{C}$ is Schnorr $P$-test.

Let $\alpha\in \Omega^\infty$. Suppose that $\alpha$ passes $\mathcal{C}$.
Assume contrarily that $\alpha$ contains some element $a_0$ of $P^{-1}(\{0\})$.
Then there exists a prefix $\sigma_0$ of $\alpha$ which contains $a_0$.
It follows that $\sigma_0\in S$, and therefore $\alpha\in\osg{\mathcal{C}_n}$ for all $n\in\N^+$.
Thus,
we have a contradiction.
Hence,
the proof of the result~(ii) is completed.
\end{proof}

\begin{theorem}\label{one_probability}
Let $\Omega$ be an alphabet, and let $P\in\PS(\Omega)$. Let $a\in\Omega$ and $\alpha\in\Omega^\infty$.
Suppose that $\alpha$ is Schnorr $P$-random and $P(a)=1$.
Then $\alpha$ consists only of $a$,
i.e.,
$\alpha=aaaaaa\dotsc\dotsc$.
\end{theorem}

\begin{proof}
Let $\Omega$ be an alphabet, and let $P\in\PS(\Omega)$.
Let $a\in\Omega$.
Suppose that $\alpha$ is Schnorr $P$-random and $P(a)=1$.
Then, since $P(a)=1$ and $\sum_{x\in\Omega} P(x)=1$,
we see that $P(x)=0$ for every $x\in\Omega\setminus\{a\}$.
Hence,
it follows from the result~(i) of Theorem~\ref{zero_probability} that
$\alpha$ does not contain $x$ for every $x\in\Omega\setminus\{a\}$.
This implies that $\alpha$ consists only of $a$, as desired.
\end{proof}

\begin{theorem}\label{contraction}
Let $\Omega$ be an alphabet, and let $P\in\PS(\Omega)$.
Let $\alpha$ be a Schnorr $P$-random infinite sequence over $\Omega$, and
let $a$ and $b$ be distinct elements
of
$\Omega$.
Suppose that $\beta$ is an infinite sequence
over $\Omega\setminus\{b\}$
obtained by replacing all occurrences of $b$ by $a$ in $\alpha$.
Then $\beta$ is
Schnorr $Q$-random,
where $Q\in\PS(\Omega\setminus\{b\})$ such that $Q(x):=P(a)+P(b)$ if $x=a$ and $Q(x):=P(x)$ otherwise.
\end{theorem}

\begin{proof}
We show the contraposition.
Suppose that $\beta$ is not Schnorr $Q$-random,
Then
it follows from Lemma~\ref{impose-prefix-freeness} that
there exists a
Schnorr
$Q$-test $\mathcal{S}\subset\N^+\times(\Omega\setminus\{b\})^*$
such that
$\mathcal{S}_n$ is a prefix-free subset of $\Omega^*$ and
\begin{equation}\label{betainosgmSn_LemmaLLN}
  \beta\in\osg{\mathcal{S}_n}
\end{equation}
for every $n\in\N^+$.
For each $\sigma\in(\Omega\setminus\{b\})^*$,
let $F(\sigma)$ be the set of all $\tau\in\Omega^*$ such that
$\tau$ is obtained by replacing some or none of the occurrences of $a$ in $\sigma$, if exists, by $b$.
Note that if $\sigma$ has exactly $n$ occurrences of $a$ then $\#F(\sigma)=2^n$.
Then, since $Q(a)=P(a)+P(b)$,
using \eqref{pBm}
we have that
\begin{equation}\label{BmQ=Q=P=BmP_LemmaLLN}
  \Bm{Q}{\osg{\sigma}}=Q(\sigma)=P(F(\sigma))=\Bm{P}{\osg{F(\sigma)}}
\end{equation}
for each $\sigma\in(\Omega\setminus\{b\})^*$.
We then define $\mathcal{T}$ to be a subset of $\N^+\times \Omega^*$ such that
$\mathcal{T}_n=\bigcup_{\sigma\in\mathcal{S}_n} F(\sigma)$ for every $n\in\N^+$.
For each $n\in\N^+$, we
see that
\begin{equation}\label{eq:BmPTn=SnBmPF=SnBmQs=BmQSn}
  \Bm{P}{\osg{\mathcal{T}_n}}
  =\sum_{\sigma\in\mathcal{S}_n}\Bm{P}{\osg{F(\sigma)}}
  =\sum_{\sigma\in\mathcal{S}_n}\Bm{Q}{\osg{\sigma}}
  =\Bm{Q}{\osg{\mathcal{S}_n}},
\end{equation}
where the first and
last
equalities follow from the prefix-freeness of $\mathcal{S}_n$ and
the second equality follows from \eqref{BmQ=Q=P=BmP_LemmaLLN}.
Moreover, since $\mathcal{S}$ is r.e., $\mathcal{T}$ is also r.e.
Thus, since $\mathcal{S}$ is a Martin-L\"of $Q$-test,
it follows from \eqref{eq:BmPTn=SnBmPF=SnBmQs=BmQSn}
that $\mathcal{T}$ is a Martin-L\"of $P$-test.
Furthermore,
since $\mathcal{S}$ is a Schnorr $Q$-test,
it follows again from \eqref{eq:BmPTn=SnBmPF=SnBmQs=BmQSn} that
$\Bm{P}{\osg{\mathcal{T}_n}}$ is uniformly compuatble in $n$.
Thus, $\mathcal{T}$ is a Schnorr $P$-test.

On the other hand,
note that, for every $n\in\N^+$, if $\beta\in\osg{\mathcal{S}_n}$ then $\alpha\in\osg{\mathcal{T}_n}$.
Thus, it follows from \eqref{betainosgmSn_LemmaLLN}
that $\alpha\in\osg{\mathcal{T}_n}$ for every $n\in\N^+$.
Hence, $\alpha$ is not Schnorr $P$-random.
This completes the proof.
\end{proof}

Note that
we do not require the underlying finite probability space $P$ to be computable at all
in Theorems~\ref{zero_probability}, \ref{one_probability}, and \ref{contraction}.

In order to prove Theorem~\ref{EffectiveLLN0}, we
also
need the following theorem, Chernoff bound.
This form of Chernoff bound follows from Theorem~4.2 of Motwani and Raghavan~\cite{MR95}.

\begin{theorem}[Chernoff bound]\label{Chernoff_bound}
Let $P\in\PS(\{0,1\})$ with $0<P(1)<1$,
and let
$\varepsilon\in\R$
with $0<\varepsilon\le \min\{P(0),P(1)\}$.
Then, for every $n\in\N^+$, we have
$$\Bm{P}{\osg{S_n}}<2\exp(-\varepsilon^2 n/2),$$
where $S_n$ is the set of all $\sigma\in\{0,1\}^n$ such that
$\abs{N_1(\sigma)/n-P(1)}>\varepsilon$.
\qed
\end{theorem}

Moreover, we need the following lemma to prove Theorem~\ref{EffectiveLLN0}.
This lemma is Lemma~53 of Tadaki~\cite{T21arXiv}.

\begin{lemma}[Tadaki~\cite{T21arXiv}]\label{EffectiveLLN-lemma}
Let $\epsilon$ be a positive rational, and let $L\in\N^+$.
Then there exists a total recursive function $g\colon\N^+\to\N^+$ such that
for every $m\in\N^+$ it holds that $g(m)\ge L$ and
\begin{equation}\label{EffectiveLLN-lemma-Result}
  \sum_{n=g(m)}^\infty\sum_{k=f(n)}^\infty \exp(-k/n^2)<2^{-m-1},
\end{equation}
where $f$ denotes a function  $f\colon\N^+\to\N^+$ defined by
$f(n)=\left\lceil n^{2+\epsilon} \right\rceil$.
\qed
\end{lemma}

For the proof of Lemma~\ref{EffectiveLLN-lemma}, see Tadaki~\cite{T21arXiv}.
Then Theorem~\ref{EffectiveLLN0} is proved as follows.
This proof is an elaboration of the proof of Theorem~51 of Tadaki~\cite{T21arXiv},
which is the corresponding original result
for
Martin-L\"of $P$-randomness
in Tadaki~\cite{T21arXiv}.

\begin{proof}[Proof of Theorem~\ref{EffectiveLLN0}]
Let $\Omega$ be an alphabet, and let $P\in\PS(\Omega)$.
Suppose that $P$ is computable. Let $\alpha\in\Omega^\infty$.
Suppose that $\alpha$ is Schnorr $P$-random.
Let $\varepsilon$ be an arbitrary positive real.
Let $a$ be an arbitrary element of $\Omega$.
We first show the following statement:
There exists $M\in\N^+$ such that for every $n\ge M$ and every $k\in\N^+$ if $k\ge n^{2+\varepsilon}$ then
\begin{equation}\label{EffectiveLLN-eq01}
  \abs{\frac{N_a(\rest{\alpha}{k})}{k}-P(a)}<\frac{1}{n}.
\end{equation}

In the case of $P(a)=0$, the statement~\eqref{EffectiveLLN-eq01} follows immediately from
the result~(i)
of Theorem~\ref{zero_probability}.
In the case of $P(a)=1$,
the statement~\eqref{EffectiveLLN-eq01} follows immediately from Theorem~\ref{one_probability}.
Thus we assume that $0<P(a)<1$, in what follows.

We define
$Q\in\PS(\{0,1\})$ by the condition that
$Q(1):=P(a)$ and $Q(0):=1-P(a)$.
Then
$0<Q(1)<1$.
Let $\beta$ be the infinite binary sequence obtained from $\alpha$
by replacing all $a$ by $1$ and
all other elements of $\Omega$
by $0$ in $\alpha$.
Then, by using Theorem~\ref{contraction}
repeatedly, it is easy to show that
$\beta$ is Schnorr $Q$-random and
$N_1(\rest{\beta}{k})=N_a(\rest{\alpha}{k})$ for every $k\in\N^+$.
Note also that $Q(1)$ is a computable real
and $Q$ is a computable finite probaility space,
since $P$ is a computable finite probability space.

We choose any specific $n_0\in\N^+$ such that
\[
  \frac{2}{n_0}\le \min\{Q(0),Q(1)\}.
\]
This is possible since $0<Q(1)<1$.
Then, it follows from Theorem~\ref{Chernoff_bound} that
\begin{equation}\label{EffectiveLLN-rCb}
  \Bm{Q}{\osg{\{\sigma\in\{0,1\}^k\mid\abs{N_1(\sigma)/k-Q(1)}>2/n\}}}<2\exp(-k/n^2)
\end{equation}
for every $n\ge n_0$ and every $k\in\N^+$.
We choose any specific positive rational $\delta$ such that $\varepsilon\ge 2\delta$, and
define a function $f\colon\N^+\to\N^+$ by $f(n):=\left\lceil n^{2+\delta} \right\rceil$.
Then it follows from Lemma~\ref{EffectiveLLN-lemma} that
there exists a total recursive function $g\colon\N^+\to\N^+$ such that
for every $m\in\N^+$ it holds that $g(m)\ge n_0$ and
\begin{equation}\label{EffectiveLLN0-lemma-Result-referred}
  \sum_{n=g(m)}^\infty\sum_{k=f(n)}^\infty \exp(-k/n^2)<2^{-m-1}.
\end{equation}
For each $m\in\N^+$,
we define a subset $S(m)$ of $\X$ by
\begin{equation}\label{EffectiveLLN0-def-S(N)}
  S(m):=\bigcup_{n=g(m)}^\infty\bigcup_{k=f(n)}^\infty\{\sigma\in\{0,1\}^k\mid\abs{N_1(\sigma)/k-Q(1)}>2/n\}.
\end{equation}
Then, for each $m\in\N^+$, we see that
\begin{equation}\label{EffectiveLLN-ineq-lQ<ss2xk2n2<2N6}
\begin{split}
  \Bm{Q}{\osg{S(m)}}
  &\le\sum_{n=g(m)}^\infty\sum_{k=f(n)}^\infty
         \Bm{Q}{\osg{\{\sigma\in\{0,1\}^k\mid\abs{N_1(\sigma)/k-Q(1)}>2/n\}}} \\
  &<\sum_{n=g(m)}^\infty\sum_{k=f(n)}^\infty 2\exp(-k/n^2) \\
  &<2^{-m},
\end{split}
\end{equation}
where the second inequality follows from the inequality~\eqref{EffectiveLLN-rCb} and the fact that $g(m)\ge n_0$,
and the last inequality follows from the inequality~\eqref{EffectiveLLN0-lemma-Result-referred}.

Now, we denote
the set
$\{(m,\sigma)\in\N^+\times\X\mid\sigma\in S(m)\}$
by $\mathcal{T}$.
We will show that $\mathcal{T}$ is a Schnorr $Q$-test.
For that purpose, we first show that $\mathcal{T}$ is a Martin-L\"{o}f $Q$-test.
On the one hand,
it is easy to see that
the function $f\colon\N^+\to\N^+$ is a total recursive function.
Moreover, note that
\[
  S(m)=\bigcup_{n=g(m)}^\infty\bigcup_{k=f(n)}^\infty
  \{\sigma\in\{0,1\}^k\mid N_1(\sigma)/k+2/n<Q(1)\;\text{ or }\;Q(1)<N_1(\sigma)/k-2/n\}
\]
for every
$m\in\N^+$.
Thus, since $Q(1)$ is
a computable real,
it follows that
$\mathcal{T}$ is an r.e.~set.
On the other hand,
\eqref{EffectiveLLN-ineq-lQ<ss2xk2n2<2N6} implies that
$\Bm{Q}{\osg{\mathcal{T}_m}}<2^{-m}$ for every $m\in\N^+$,
where $\mathcal{T}_m$ denotes the set
$\left\{\,
  \sigma\mid (m,\sigma)\in\mathcal{T}
\,\right\}$.
Hence, $\mathcal{T}$ is a Martin-L\"{o}f $Q$-test.

We will then show that $\Bm{Q}{\osg{\mathcal{T}_m}}$ is uniformly compuatble in $m$.
For that purpose, we show the following three fact.
First, note that the total recursive function $g\colon\N^+\to\N^+$ is unbounded
since it satisfies \eqref{EffectiveLLN0-lemma-Result-referred}.
Therefore it follows again from \eqref{EffectiveLLN0-lemma-Result-referred} that
there exists a total recursive function $p\colon\N^+\times\N\to\N^+$ such that
$g(p(m,l))\ge g(m)$ and
\begin{equation}\label{Schnorr-EffectiveLLN-sumsum}
  \sum_{n=g(p(m,l))}^\infty\sum_{k=f(n)}^\infty \exp(-k/n^2)<2^{-l-3}
\end{equation}
for every $m\in\N^+$ and $l\in\N$.
Thus, for each $m\in\N^+$ and $l\in\N$,
we have that
\begin{equation}\label{Schnorr-EffectiveLLN-ineq2}
\begin{split}
  &\Bm{Q}{\osg{\bigcup_{n=g(p(m,l))}^\infty\bigcup_{k=f(n)}^\infty
  \{\sigma\in\{0,1\}^k\mid\abs{N_1(\sigma)/k-Q(1)}>2/n\}}} \\
  &\le\sum_{n=g(p(m,l))}^\infty\sum_{k=f(n)}^\infty
  \Bm{Q}{\osg{\{\sigma\in\{0,1\}^k\mid\abs{N_1(\sigma)/k-Q(1)}>2/n\}}} \\
  &<\sum_{n=g(p(m,l))}^\infty\sum_{k=f(n)}^\infty 2\exp(-k/n^2) \\
  &<2^{-l-2},
\end{split}
\end{equation}
where the second inequality follows from the inequality~\eqref{EffectiveLLN-rCb} and the fact that $g(p(m,l))\ge g(m)\ge n_0$,
and the last inequality follows from the inequality~\eqref{Schnorr-EffectiveLLN-sumsum}.

Second,
for each $n\in\N^+$,
using the mean value theorem, we have that
\begin{equation}\label{Schnorr-EffectiveLLN-MVT1}
  1-\exp(-1/n^2)>\exp(-1/n^2)\frac{1}{n^2}.
\end{equation}
Thus, for each $L\in\N^+$, we
have
that
\begin{equation}\label{Schnorr-EffectiveLLN-ineq01}
  \sum_{k=L+1}^\infty \exp(-k/n^2)
  =\frac{\exp\{-(L+1)/n^2\}}{1-\exp(-1/n^2)}
  <n^2\exp(-L/n^2),
\end{equation}
where the last inequality follows from the inequality~\eqref{Schnorr-EffectiveLLN-MVT1}.
It is easy to see that
there exists a total recursive function $q\colon\N^+\times\N\to\N^+$ such that
\begin{equation}\label{Schnorr-EffectiveLLN-def01}
  g(p(m,l))^3\exp\{-q(m,l)/g(p(m,l))^2\}<2^{-l-3}
\end{equation}
for every $m\in\N^+$ and $l\in\N$.
Thus, for each $m\in\N^+$ and $l\in\N$,
we have that
\begin{equation}\label{Schnorr-EffectiveLLN-ineq3}
\begin{split}
  &\Bm{Q}{\osg{\bigcup_{n=g(m)}^{g(p(m,l))}\bigcup_{k=q(m,l)+1}^\infty
  \{\sigma\in\{0,1\}^k\mid\abs{N_1(\sigma)/k-Q(1)}>2/n\}}} \\
  &\le\sum_{n=g(m)}^{g(p(m,l))}\sum_{k=q(m,l)+1}^\infty
  \Bm{Q}{\osg{\{\sigma\in\{0,1\}^k\mid\abs{N_1(\sigma)/k-Q(1)}>2/n\}}} \\
  &<\sum_{n=g(m)}^{g(p(m,l))}\sum_{k=q(m,l)+1}^\infty 2\exp(-k/n^2) \\
  &<2^{-l-2},
\end{split}
\end{equation}
where the second inequality follows from the inequality~\eqref{EffectiveLLN-rCb} and the fact that $g(h(m,l))\ge g(m)\ge n_0$,
and the last inequality follows from the inequalities~\eqref{Schnorr-EffectiveLLN-ineq01}
and \eqref{Schnorr-EffectiveLLN-def01}.

Third, since $Q(1)$ is
computable,
it follows that
one can effectively calculate the finite set
$\{\sigma\in\{0,1\}^k\mid\abs{N_1(\sigma)/k-Q(1)}>2/n\}$,
given $n\in\N^+$ and $k\in\N^+$.
This is obvious in the case of $Q(1)\in\Q$.
This also holds true in the case of $Q(1)\notin\Q$
because, in such a case,
either $\abs{N_1(\sigma)/k-Q(1)}>2/n$ or $\abs{N_1(\sigma)/k-Q(1)}<2/n$ 
holds for each $n\in\N^+$, $k\in\N^+$, and $\sigma\in\{0,1\}^k$,
and moreover $Q(1)$ is computable.
Thus, given $m\in\N^+$ and $l\in\N$, one can effectively calculate the finite set
\[
  \bigcup_{n=g(m)}^{g(p(m,l))}\bigcup_{k=f(n)}^{q(m,l)}
  \{\sigma\in\{0,1\}^k\mid\abs{N_1(\sigma)/k-Q(1)}>2/n\}.
\]
Hence,
since $Q$ is a computable finite probability space,
it is easy to show that
there exists a total recursive function $r\colon\N^+\times\N\to\Q$ such that
\begin{equation}\label{Schnorr-EffectiveLLN-Approx}
  \abs{\Bm{Q}{\osg{\bigcup_{n=g(m)}^{g(p(m,l))}\bigcup_{k=f(n)}^{q(m,l)}
  \{\sigma\in\{0,1\}^k\mid\abs{N_1(\sigma)/k-Q(1)}>2/n\}}}
  -r(m,l)}<2^{-l-2}
\end{equation}
for every $m\in\N^+$ and $l\in\N$.

Now, from the definition~\eqref{EffectiveLLN0-def-S(N)}, we see that
\begin{align*}
  S(m)=&\bigcup_{n=g(m)}^{g(p(m,l))}\bigcup_{k=f(n)}^{q(m,l)}\{\sigma\in\{0,1\}^k\mid\abs{N_1(\sigma)/k-Q(1)}>2/n\} \\
  &\cup\bigcup_{n=g(m)}^{g(p(m,l))}\bigcup_{k=q(m,l)+1}^\infty\{\sigma\in\{0,1\}^k\mid\abs{N_1(\sigma)/k-Q(1)}>2/n\} \\
  &\cup\bigcup_{n=g(p(m,l))}^\infty\bigcup_{k=f(n)}^\infty\{\sigma\in\{0,1\}^k\mid\abs{N_1(\sigma)/k-Q(1)}>2/n\}
\end{align*}
for every $m\in\N^+$ and $l\in\N$.
Therefore, it follows from \eqref{Schnorr-EffectiveLLN-ineq2},
\eqref{Schnorr-EffectiveLLN-ineq3},
and \eqref{Schnorr-EffectiveLLN-Approx} that
\begin{equation*}%
  \abs{\Bm{Q}{\osg{S(m)}}-r(m,l)}<2^{-l}
\end{equation*}
for every $m\in\N^+$ and $l\in\N$.
Thus, $\Bm{Q}{\osg{\mathcal{T}_m}}$ is uniformly compuatble in $m$.
Hence, $\mathcal{T}$ is a Schnorr $Q$-test.

Thus,
since $\beta$ is Schnorr $Q$-random and $\mathcal{T}$ is a Schnorr $Q$-test,
we have that
there exists $L\in\N^+$ such that $\beta\notin\osg{\mathcal{T}_L}$.
Thus,
since $\mathcal{T}_L=S(L)$,
from the definition of $S(L)$
we have the following:
For every $n\ge g(L)$ and every $k\ge f(n)$ it holds that
\begin{equation}\label{EffectiveLLN0-eq-absNbok-Q1le2on}
  \abs{\frac{N_1(\rest{\beta}{k})}{k}-Q(1)}\le\frac{2}{n}.
\end{equation}
Then we choose any specific $M\in\N^+$ such that $M^\delta\ge 4^{2+\delta}$ and $4M\ge g(L)$.
Note that for every $n\in\N^+$ if $n\ge M$ then
$4n\ge g(L)$ and
$\left\lceil n^{2+\varepsilon}\right\rceil\ge f(4n)$.
It follows from \eqref{EffectiveLLN0-eq-absNbok-Q1le2on} that
for every $n\ge M$ and every $k\ge n^{2+\varepsilon}$ it holds that
\[
  \abs{\frac{N_1(\rest{\beta}{k})}{k}-Q(1)}<\frac{1}{n}.
\]
Thus, since $Q(1)=P(a)$ and $N_1(\rest{\beta}{k})=N_a(\rest{\alpha}{k})$ for every $k\in\N^+$,
the statement~\eqref{EffectiveLLN-eq01} holds in this case of $0<P(a)<1$, as desired.
Hence, since $a$ is an arbitrary element of $\Omega$,
the statement~\eqref{EffectiveLLN-eq01} holds for every $a\in\Omega$.

Now, for each $a\in\Omega$,
let $M_a$ be the positive integer $M$ whose existence is guaranteed in the statement~\eqref{EffectiveLLN-eq01}.
Note that $\max\{M_a\mid a\in\Omega\}$ exists, since $\Omega$ is a
non-empty
finite set. 
We denote
by $\overline{M}$
this $\max\{M_a\mid a\in\Omega\}$.
It is then easy to check that
for every $n\ge \overline{M}$ and every $k\in\N^+$ if $k\ge n^{2+\varepsilon}$ then for every $a\in\Omega$
the inequality~\eqref{EffectiveLLN0-eq} holds.
This completes the proof.
\end{proof}

\subsection{The proof of Theorem~\ref{SpeedLimitTheorem}}
\label{Proof-of-TheoremSpeedLimitTheorem}

In this subsection, we prove Theorem~\ref{SpeedLimitTheorem}.
The central limit theorem plays a key role
in its proof.

\begin{proof}[Proof of Theorem~\ref{SpeedLimitTheorem}]
Let $\Omega$ be an alphabet, and let $P\in\PS(\Omega)$.
Let $\alpha\in\Omega^\infty$ and let $a\in\Omega$.
Suppose that $0<P(a)<1$.
For each $n\in \N^+$, let $X_n$ be a random variable on
the measurable space $(\Omega^\infty,\mathcal{B}_{\Omega})$
such that $X_n(\beta)=1$ if $\beta(n)=a$ and $X_n(\beta)=0$ otherwise
for every $\beta\in\Omega^\infty$.
For each $n\in \N^+$, let $S_n:=X_1+\dots +X_n$.
Then, since $\lambda_P$ is a Bernoulli measure on
$(\Omega^\infty,\mathcal{B}_{\Omega})$,
we see that
$X_1,X_2,\dotsc$ are independent and identically distributed random variables on
the probability space~$(\Omega^\infty,\mathcal{B}_{\Omega},\lambda_P)$.
Moreover, the variance of $X_1$ equals $P(a)(1-P(a))$ and is therefore positive,
since $0<P(a)<1$.
Thus, it follows from the central limit theorem
(see Theorem~\ref{CLT} in Section~\ref{Section:EffectiveSLLN}) that
for every reals $l_1$ and $l_2$ with $l_1<l_2$ it holds that
\begin{equation*}%
  \lim_{n\to\infty}\Bm{P}{l_1\sqrt{npq}\le S_n-np\le l_2\sqrt{npq}}
  =\int_{l_1}^{l_2} \frac{1}{\sqrt{2\pi}}e^{-x^2/2}dx,
\end{equation*}
where $p$ denotes $P(a)$ and $q$ denotes $1-p$.
Note that $pq>0$.
It follows that
\begin{equation*}%
  \lim_{n\to\infty}\Bm{P}{-3\cdot 2^n\le S_{3\cdot 4^n}-3\cdot 4^n p\le 3\cdot 2^n}
  =\int_{-\sqrt{3/(pq)}}^{\sqrt{3/(pq)}}\frac{1}{\sqrt{2\pi}}e^{-x^2/2}dx<1.
\end{equation*}
Therefore there exist
a positive integer $n_0$
and a rational $r\in(0,1)$ such that
\begin{equation}\label{pr:SpeedLimitTheorem:ieq1}
  \Bm{P}{\abs{S_{3\cdot 4^n}-3\cdot 4^n p}\le 3\cdot 2^n}<r
\end{equation}
for all $n\ge n_0$.

On the other hand, by the assumption of the theorem, we have that
there exists
$n_1\in\N^+$ with $n_1\ge n_0$
such that
\begin{equation}\label{pr:SpeedLimitTheorem:eq4}
  \abs{N_a(\rest{\alpha}{4^n})-4^n p}\le 2^n
\end{equation}
for all $n\ge n_1$.

Now, let $n$ be an arbitrary integer with $n\ge n_1$.
First, note that, for each $\beta\in\Omega^\infty$,
if $\abs{S_{4^n}(\beta)-4^{n}p}\le 2^{n}$ and
$\abs{S_{4^{n+1}}(\beta)-4^{n+1}p}\le 2^{n+1}$ then
\[
  \abs{(X_{4^n+1}+\dots +X_{4^{n+1}})(\beta)-3\cdot 4^{n}p}
  =\abs{(S_{4^{n+1}}(\beta)-S_{4^n}(\beta))-3\cdot 4^{n}p}
  \le 3\cdot 2^{n}.
\]
Therefore, we have that
\begin{equation}\label{pr:SpeedLimitTheorem:eqs2}
\begin{split}
  &\Bm{P}{\bigwedge_{k=n_1}^{n+1}\abs{S_{4^k}-4^k p}\le 2^{k}} \\
  &\le \Bm{P}{\bigwedge_{k=n_1}^{n}\abs{S_{4^k}-4^k p}\le 2^{k}\;\&\;
    \abs{X_{4^n+1}+\dots +X_{4^{n+1}}-3\cdot 4^{n} p}\le 3\cdot 2^{n}} \\
  &=\Bm{P}{\bigwedge_{k=n_1}^{n}\abs{S_{4^k}-4^k p}\le 2^{k}}
    \Bm{P}{\abs{X_{4^n+1}+\dots +X_{4^{n+1}}-3\cdot 4^{n} p}\le 3\cdot 2^{n}} \\
  &=\Bm{P}{\bigwedge_{k=n_1}^{n}\abs{S_{4^k}-4^k p}\le 2^{k}}
    \Bm{P}{\abs{S_{3\cdot 4^n}-3\cdot 4^n p}\le 3\cdot 2^n} \\
  &\le \Bm{P}{\bigwedge_{k=n_1}^{n}\abs{S_{4^k}-4^k p}\le 2^{k}}r,
\end{split}
\end{equation}
where the symbol~$\wedge$, as well as $\&$, denotes the logical conjunction, and
the first and second equalities follow from the
fact
that
$X_1,X_2,\dotsc$ are independent and identically distributed random variables
on $(\Omega^\infty,\mathcal{B}_{\Omega},\lambda_P)$ and
the last  inequality follows from \eqref{pr:SpeedLimitTheorem:ieq1}.
Thus, since $n$ is an arbitrary integer with $n\ge n_1$, it follows from \eqref{pr:SpeedLimitTheorem:eqs2} that
\begin{equation}\label{pr:SpeedLimitTheorem:eq5}
  \Bm{P}{\bigwedge_{k=n_1}^{n}\abs{S_{4^k}-4^kp}\le 2^{k}}
  \le\Bm{P}{\bigwedge_{k=n_1}^{n_1}\abs{S_{4^k}-4^kp}\le 2^{k}}r^{n-n_1} \le r^{n-n_1}
\end{equation}
for each $n\ge n_1$.

Since $r$ is a rational with $0<r<1$, it is easy to show that
there exists a total recursive function $f\colon\N^+\to\N^+$ such that
$f(m)\ge n_1$ and $r^{f(m)-n_1}<2^{-m}$ for all $m\in\N^+$.
For each $m\in\N^+$, we define a subset $S(m)$ of $\Omega^*$
as the set of all finite strings $\sigma$
of  length~$4^{f(m)}$ such that
for every $k\in\N^+$ with $n_1\le k\le f(m)$ there exists a prefix~$\tau$ of $\sigma$ for which
$\abs{\tau}=4^k$ and $\abs{N_a(\tau)-4^k p}\le 2^k$ hold.
Then, for each $m\in\N^+$, it is easy to see that
\begin{equation}\label{pr:SpeedLimitTheorem:eq7}
  \osg{S(m)}=\bigcap_{k=n_1}^{f(m)}
  \osg{\left\{\sigma\in\Omega^*\,\middle\vert\,\abs{\sigma}=4^k\;\&\;\abs{N_a(\sigma)-4^k p}\le 2^k\right\}},
\end{equation}
and therefore using \eqref{pr:SpeedLimitTheorem:eq5} we have that
\begin{equation}\label{pr:SpeedLimitTheorem:eq6}
\begin{split}
  \Bm{P}{\osg{S(m)}}&=\Bm{P}{\bigcap_{k=n_1}^{f(m)}
  \osg{\left\{\sigma\in\Omega^*\,\middle\vert\,\abs{\sigma}=4^k\;\&\;\abs{N_a(\sigma)-4^k p}\le 2^k\right\}}} \\
  &=\Bm{P}{\bigwedge_{k=n_1}^{f(m)}\abs{S_{4^k}-4^kp}\le 2^{k}} \\
  &\le r^{f(m)-n_1}<2^{-m}.
\end{split}
\end{equation}
We denote the set~$\{(m,\sigma)\in\N^+\times\X\mid\sigma\in S(m)\}$ by $\mathcal{T}$.
Then
\eqref{pr:SpeedLimitTheorem:eq6} implies that
\begin{equation}\label{pr:SpeedLimitTheorem:eq8}
  \Bm{P}{\osg{\mathcal{T}_m}}<2^{-m}
\end{equation}
for every $m\in\N^+$,
where $\mathcal{T}_m$ denotes the set
$\left\{\,
  \sigma\mid (m,\sigma)\in\mathcal{T}
\,\right\}$.

Now, suppose that $P$ is computable.
Then, since
$p$, i.e., $P(a)$, is a computable real and
the function $f\colon\N^+\to\N^+$ is a total recursive function,
it is easy to see that one can effectively enumerate all the elements of the finite set $S(m)$, given $m\in\N^+$.
This
fact
is obvious in the case of $p\in\Q$.
In the case of $p\notin\Q$,
this
fact
can also be verified by noting the facts that
either $\abs{N_a(\sigma)-4^k p}< 2^k$ or $\abs{N_a(\sigma)-4^k p}> 2^k$ holds for
each $k\in\N^+$ and $\sigma\in\Omega^*$,
and moreover $p$ is computable.
Hence,
first of all,
$\mathcal{T}$ is r.e., and therefore it follows
from \eqref{pr:SpeedLimitTheorem:eq8} that $\mathcal{T}$ is Martin-L\"of $P$-test.
Moreover, since
$P$ is
a computable finite probability space,
it follows that
$\Bm{P}{\osg{S(m)}}$, i.e., $\Bm{P}{\osg{\mathcal{T}_m}}$,
is uniformly computable in $m$,
and thus $\mathcal{T}$ is Schnorr $P$-test.

Finally, using \eqref{pr:SpeedLimitTheorem:eq4} and \eqref{pr:SpeedLimitTheorem:eq7}
we have that $\alpha\in\osg{\mathcal{T}_m}$ for all $m\in\N^+$,
and therefore $\alpha$ is not Schnorr  $P$-random, as desired.
This completes the proof.
\end{proof}

\section{Effectivization of the law of large numbers for a real random variable, and its absolute speed limit of convergence}
\label{Section:EffectiveLLN=RV}

In this section,
we investigate an effectivization of the law of large numbers,
and its absolute speed limit of convergence,
for a real random variable
whose underlying process
is described as
an arbitrary Schnorr $P$-random infinite sequence,
where $P$ is
an arbitrary
computable finite probability space.

Let $\Omega$ be an alphabet, and let $P\in\PS(\Omega)$.
A \emph{real random variable} on $\Omega$ is a real function $X\colon\Omega\to\R$.
For each real random variable $X$ on $\Omega$,
we define the \emph{mean} $E(X)$ of $X$ by
\[
  E(X):=\sum_{a\in\Omega}X(a)P(a),
\]
and define the \emph{variance} $V(X)$ of $X$ by
\[
  V(X):=\sum_{a\in\Omega}(X(a)-E(X))^2 P(a).
\]

First, the following theorem
is a generalization of
Theorem~\ref{EffectiveLLN0}
over the law of large numbers for a real random variable.

\begin{theorem}
[Effectivization of the law of large numbers for a real random variable I]
\label{EffectiveLLN-RV}
Let $\Omega$ be an alphabet, and let $P\in\PS(\Omega)$.
Suppose that $P$ is computable.
Let $X$ be an arbitrary real random variable on $\Omega$.
Let $\alpha\in\Omega^\infty$. Suppose that $\alpha$ is Schnorr $P$-random.
Let $\varepsilon$ be an arbitrary positive real.
Then
there exists $M\in\N^+$ such that for every $n\ge M$ and every $k\in\N^+$ if $k\ge n^{2+\varepsilon}$ then
\begin{equation*}%
  \abs{\frac{1}{k}\sum_{i=1}^k X(\alpha(i))-E(X)}<\frac{1}{n}.
\end{equation*}
\end{theorem}

\begin{proof}
We choose any specific $L\in\N^+$ such that
$L\ge \abs{X(a)}$ for every $a\in\Omega$.
This can be possible since $\Omega$ is a finite set.
Then, for each $k\in\N^+$, we see that
\begin{equation}\label{Schnorr-EffectiveLLN-RV:eq1-0}
\begin{split}
\abs{\frac{1}{k}\sum_{i=1}^k X(\alpha(i))-E(X)}
&=\abs{\frac{1}{k}\sum_{a\in\Omega}N_a(\rest{\alpha}{k})X(a)-E(X)}
\le\sum_{a\in\Omega}\abs{\frac{N_a(\rest{\alpha}{k})}{k}-P(a)}\abs{X(a)} \\
&\le L\sum_{a\in\Omega}\abs{\frac{N_a(\rest{\alpha}{k})}{k}-P(a)},
\end{split}
\end{equation}
where $N_a(\rest{\alpha}{k})$ denotes the number of the occurrences of $a$ in $\rest{\alpha}{k}$
for every $a\in\Omega$ as before.

Now, we choose any specific positive real $\delta$ such that $\varepsilon\ge 2\delta$.
Then it follows from Theorem~\ref{EffectiveLLN0} that
there exists $M_0\in\N^+$ such that for every $n\ge M_0$ and every $k\in\N^+$ if $k\ge n^{2+\delta}$ then
for every $a\in\Omega$ it holds that
\begin{equation}\label{Schnorr-EffectiveLLN0-eq-RV}
  \abs{\frac{N_a(\rest{\alpha}{k})}{k}-P(a)}<\frac{1}{n}.
\end{equation}
Then we choose any specific $M\in\N^+$ such that
$M^\delta\ge (L\#\Omega)^{2+\delta}$ and $ML\#\Omega\ge M_0$.
Note that for every $n\in\N^+$ if $n\ge M$ then
$nL\#\Omega\ge M_0$ and
$n^{2+\varepsilon}\ge (nL\#\Omega)^{2+\delta}$.
Thus, for each $n\ge M$ and each $k\in\N^+$,
using \eqref{Schnorr-EffectiveLLN-RV:eq1-0} and \eqref{Schnorr-EffectiveLLN0-eq-RV}
we see that if $k\ge n^{2+\varepsilon}$ then
\[
  \abs{\frac{1}{k}\sum_{i=1}^k X(\alpha(i))-E(X)}
  \le L\sum_{a\in\Omega}\abs{\frac{N_a(\rest{\alpha}{k})}{k}-P(a)}
  < L\#\Omega\frac{1}{nL\#\Omega}=\frac{1}{n}.
\]
This completes the proof.
\end{proof}

We obtain the following corollary from Theorem~\ref{EffectiveLLN-RV}.

\begin{corollary}
[Effectivization of the law of large numbers for a real random variable II]
\label{cor:prf-EffectiveLLN-RV}
Let $\Omega$ be an alphabet, and let $P\in\PS(\Omega)$.
Suppose that $P$ is computable.
Let $X$ be an arbitrary real random variable on $\Omega$.
Let $\alpha\in\Omega^\infty$. Suppose that $\alpha$ is Schnorr $P$-random.
Then there exists a primitive recruesive function $f\colon\N\to\N$
such that for every $n\in\N^+$ and every $k\in\N^+$ if $k\ge f(n)$ then
\begin{equation}\label{prf-Schnorr-EffectiveLLN-RV-eq}
  \abs{\frac{1}{k}\sum_{i=1}^k X(\alpha(i))-E(X)}<\frac{1}{n}.
\end{equation}
\end{corollary}

\begin{proof}
It follows from Theorem~\ref{EffectiveLLN-RV} that there exists $M\in\N^+$ such that
for every $n\ge M$ and every $k\in\N^+$ if $k\ge n^3$ then
the inequality~\eqref{prf-Schnorr-EffectiveLLN-RV-eq} holds.
We then define a function $f\colon\N\to\N$ by the condition that
$f(n):=n^3$ if $n\ge M$ and $f(n):=M^3$ otherwise.
Thus, since $f\colon\N\to\N$ is a primitive recursive function,
the result follows.
\end{proof}

As a specific application of Theorem~\ref{EffectiveLLN-RV},
we present an effectivization of the \emph{asymptotic equipartition property}
(AEP, for short).
The AEP plays an important role in the source coding problem in information theory,
and is a direct consequence of the weak law of large numbers
for independent, identically distributed random variables in probability theory.
The AEP is stated in terms of the notion of Shannon entropy.
Let $\Omega$ be an alphabet, and let $P\in\PS(\Omega)$. The \emph{Shannon entropy} $H(P)$ of $P$ is defined by
\begin{equation}\label{def:Shannon-entropy}
  H(P):=-\sum_{a\in\Omega}P(a)\log_2 P(a),
\end{equation}
where $0\log_2 0$ is defined to be $0$ as usual.
See Cover and Thomas~\cite[Chapter 3]{CT06} for the details
of the AEP and its applications, where the AEP is stated as Theorem~3.1.1.

The AEP
is
effectivized in terms of a Schnorr $P$-random sequence as follows:

\begin{theorem}[Effectivization of AEP]\label{effective-AEP}
Let $\Omega$ be an alphabet, and let $P\in\PS(\Omega)$.
Suppose that $P$ is computable.
Let $\alpha\in\Omega^\infty$. Suppose that $\alpha$ is Schnorr $P$-random.
Then the following (i) and (ii) hold:
\begin{enumerate}
\item $\Bm{P}{\osg{\rest{\alpha}{n}}}>0$ for every $n\in\N^+$.
\item
For every real $\varepsilon>0$,
there exists $M\in\N^+$ such that for every $n\ge M$ and every $k\in\N^+$ if $k\ge n^{2+\varepsilon}$ then
\begin{equation*}%
  \abs{\frac{-\log_2\Bm{P}{\osg{\rest{\alpha}{k}}}}{k}-H(P)}<\frac{1}{n}.
\end{equation*}
\end{enumerate}
\end{theorem}

\begin{proof}
We denote by $\Omega_e$ the set $\{a\in\Omega\mid P(a)>0\}$.
Since $\alpha$ is Schnorr $P$-random,
it follows from
the result~(i)
of Theorem~\ref{zero_probability} that
\begin{equation}\label{effective-AEP:eq1}
  \alpha(n)\in\Omega_e
\end{equation}
for every $n\in\N^+$.
Therefore, for each $n\in\N^+$, using~\eqref{pBm} we have that
\begin{equation}\label{effective-AEP:eq0}
  \Bm{P}{\osg{\rest{\alpha}{n}}}=P(\rest{\alpha}{n})=\prod_{k=1}^n P(\alpha(k))>0,
\end{equation}
as desired.

Now, we define a function $X\colon\Omega\to\R$ by the condition that
$X(a):=\log_2 P(a)$ if $a\in\Omega_e$ and $X(a):=0$ otherwise.
Let $k\in\N^+$.
On the one hand, using \eqref{effective-AEP:eq0}
and \eqref{effective-AEP:eq1}
we have that
\begin{equation*}%
\frac{\log_2\Bm{P}{\osg{\rest{\alpha}{k}}}}{k}=\frac{1}{k}\sum_{i=1}^k\log_2 P(\alpha(i))
=\frac{1}{k}\sum_{i=1}^k X(\alpha(i)).
\end{equation*}
On the other hand, we have that
\begin{equation*}%
-H(P)=\sum_{a\in\Omega}P(a)\log_2 P(a)=\sum_{a\in\Omega_e}P(a)\log_2 P(a)
=\sum_{a\in\Omega_e}X(a)P(a)=\sum_{a\in\Omega}X(a)P(a)
=E(X).
\end{equation*}
Thus, the result~(ii) of Theorem~\ref{effective-AEP} follows from
Theorem~\ref{EffectiveLLN-RV}.
\end{proof}

The following theorem generalizes Theorem~\ref{SpeedLimitTheorem}
over the law of large numbers for a real random variable.

\begin{theorem}
[Convergence speed limit theorem for a random variable]
\label{SpeedLimitTheorem-RV}
Let $\Omega$ be an alphabet, and let $P\in\PS(\Omega)$.
Suppose that $P$ is computable.
Let $\alpha\in\Omega^\infty$,
and let $X$ be a real random variable on $\Omega$.
Suppose that
$V(X)>0$
and there exists $n_0\in\N^+$ such that for every $n\ge n_0$ it holds that
\begin{equation*}%
  \abs{\sum_{i=1}^{4^n} X(\alpha(i))-4^n E(X)}\le 2^n.
\end{equation*}
Then the following (i) and (ii) hold:
\begin{enumerate}
  \item If $X(a)$ is a computable real for all $a\in\Omega$,
    then $\alpha$ is not Martin-L\"of $P$-random.
  \item If $X(a)$ is a rational for all $a\in\Omega$,
    then $\alpha$ is not Schnorr $P$-random.
\end{enumerate}
\end{theorem}

\begin{proof}
Let $\Omega$ be an alphabet, and let $P\in\PS(\Omega)$.
Let $\alpha\in\Omega^\infty$.
For each $n\in \N^+$, let $X_n$ be a random variable on
the measurable space $(\Omega^\infty,\mathcal{B}_{\Omega})$
such that $X_n(\beta):=X(\beta(n))$ for every $\beta\in\Omega^\infty$.
For each $n\in \N^+$, let $S_n:=X_1+\dots +X_n$.
Then, since $\lambda_P$ is a Bernoulli measure on
$(\Omega^\infty,\mathcal{B}_{\Omega})$,
we see that
$X_1,X_2,\dotsc$ are independent and identically distributed random variables on
the probability space~$(\Omega^\infty,\mathcal{B}_{\Omega},\lambda_P)$.
Moreover, the variance of $X_1$ is positive since it equals $V(X)>0$.
Thus, it follows from the central limit theorem, Theorem~\ref{CLT} below, that
for every reals $l_1$ and $l_2$ with $l_1<l_2$ it holds that
\begin{equation*}%
  \lim_{n\to\infty}\Bm{P}{l_1\sqrt{nv}\le S_n-n\mu\le l_2\sqrt{nv}}
  =\int_{l_1}^{l_2} \frac{1}{\sqrt{2\pi}}e^{-x^2/2}dx,
\end{equation*}
where $\mu:=E(X)$ and $v:=V(X)$.
Note that $v>0$.
It follows that
\begin{equation*}%
  \lim_{n\to\infty}\Bm{P}{-3\cdot 2^{n+1}\le S_{3\cdot 4^n}-3\cdot 4^n \mu\le 3\cdot 2^{n+1}}
  =\int_{-2\sqrt{3/v}}^{2\sqrt{3/v}}\frac{1}{\sqrt{2\pi}}e^{-x^2/2}dx<1.
\end{equation*}
Therefore there exist a positive integer $n_0$ and a rational $r\in(0,1)$ such that
\begin{equation}\label{pr:SpeedLimitTheorem-RV:ieq1}
  \Bm{P}{\abs{S_{3\cdot 4^n}-3\cdot 4^n \mu}< 3\cdot 2^{n+1}}<r
\end{equation}
for all $n\ge n_0$.

On the other hand, by the assumption of the theorem, we have that
there exists $n_1\in\N^+$ with $n_1\ge n_0$ such that
\begin{equation}\label{pr:SpeedLimitTheorem-RV:eq4}
  \abs{S_{4^n}(\alpha)-4^n\mu}< 2^{n+1}.
\end{equation}
for all $n\ge n_1$.

Now, let $n$ be an arbitrary integer with $n\ge n_1$.
First, note that, for each $\beta\in\Omega^\infty$,
if $\abs{S_{4^n}(\beta)-4^{n}\mu}< 2^{n+1}$ and
$\abs{S_{4^{n+1}}(\beta)-4^{n+1}\mu}< 2^{n+2}$ then
\[
  \abs{(X_{4^n+1}+\dots +X_{4^{n+1}})(\beta)-3\cdot 4^{n}\mu}
  =\abs{(S_{4^{n+1}}(\beta)-S_{4^n}(\beta))-3\cdot 4^{n}\mu}
  < 3\cdot 2^{n+1}.
\]
Therefore, we have that
\begin{equation}\label{pr:SpeedLimitTheorem-RV:eqs2}
\begin{split}
  &\Bm{P}{\bigwedge_{k=n_1}^{n+1}\abs{S_{4^k}-4^k \mu}< 2^{k+1}} \\
  &\le \Bm{P}{\bigwedge_{k=n_1}^{n}\abs{S_{4^k}-4^k \mu}< 2^{k+1}\;\&\;
    \abs{X_{4^n+1}+\dots +X_{4^{n+1}}-3\cdot 4^{n} \mu}< 3\cdot 2^{n+1}} \\
  &=\Bm{P}{\bigwedge_{k=n_1}^{n}\abs{S_{4^k}-4^k \mu}< 2^{k+1}}
    \Bm{P}{\abs{S_{3\cdot 4^n}-3\cdot 4^n \mu}< 3\cdot 2^{n+1}} \\
  &\le \Bm{P}{\bigwedge_{k=n_1}^{n}\abs{S_{4^k}-4^k \mu}< 2^{k+1}}r,
\end{split}
\end{equation}
where the symbol~$\wedge$, as well as $\&$, denotes the logical conjunction, and
the equality follows
from the
fact
that
$X_1,X_2,\dotsc$ are independent and identically distributed random variables
on $(\Omega^\infty,\mathcal{B}_{\Omega},\lambda_P)$ and
the last  inequality follows from \eqref{pr:SpeedLimitTheorem-RV:ieq1}.
Thus, since $n$ is an arbitrary integer with $n\ge n_1$, it follows from \eqref{pr:SpeedLimitTheorem-RV:eqs2} that
\begin{equation}\label{pr:SpeedLimitTheorem-RV:eq5}
  \Bm{P}{\bigwedge_{k=n_1}^{n}\abs{S_{4^k}-4^k\mu}< 2^{k+1}}
  \le\Bm{P}{\bigwedge_{k=n_1}^{n_1}\abs{S_{4^k}-4^k\mu}< 2^{k+1}}r^{n-n_1} \le r^{n-n_1}
\end{equation}
for each $n\ge n_1$.

Since $r$ is a rational with $0<r<1$, it is easy to show that
there exists a total recursive function $f\colon\N^+\to\N^+$ such that
$f(m)\ge n_1$ and $r^{f(m)-n_1}<2^{-m}$ for all $m\in\N^+$.
For each $m\in\N^+$, we define a subset $S(m)$ of $\Omega^*$
as the set of all finite strings $\sigma$
of  length~$4^{f(m)}$ such that
for every $k\in\N^+$ with $n_1\le k\le f(m)$ there exists a prefix~$\tau$ of $\sigma$ for which
$\abs{\tau}=4^k$ and
\[
  \abs{\sum_{i=1}^{4^k} X(\tau(i))-4^k \mu}< 2^{k+1}
\]
hold.
Then, for each $m\in\N^+$, it is easy to see that
\begin{equation}\label{pr:SpeedLimitTheorem-RV:eq7}
  \osg{S(m)}=\bigcap_{k=n_1}^{f(m)}
  \osg{\left\{\sigma\in\Omega^*\,\middle\vert\,\abs{\sigma}=4^k\;\&\;\abs{\sum_{i=1}^{4^k} X(\sigma(i))-4^k \mu}< 2^{k+1}\right\}}.
\end{equation}
We denote the set~$\{(m,\sigma)\in\N^+\times\X\mid\sigma\in S(m)\}$ by $\mathcal{T}$.
Then, for each $m\in\N^+$,
we have that
\begin{equation}\label{pr:SpeedLimitTheorem-RV:eq8}
  \Bm{P}{\osg{\mathcal{T}_m}}
  =\Bm{P}{\osg{S(m)}}
  =\Bm{P}{\bigwedge_{k=n_1}^{f(m)}\abs{S_{4^k}-4^k\mu}< 2^{k+1}}
  \le r^{f(m)-n_1}<2^{-m},
\end{equation}
where $\mathcal{T}_m$ denotes the set
$\left\{\,
  \sigma\mid (m,\sigma)\in\mathcal{T}
\,\right\}$, and
the second equality follows from \eqref{pr:SpeedLimitTheorem-RV:eq7} and
the first inequality follows from \eqref{pr:SpeedLimitTheorem-RV:eq5}.

Now, suppose that $P$ is computable and $X(a)$ is a computable real for all $a\in\Omega$.
Then $\mu$ is a computable real, obviously.
Thus, since the function $f\colon\N^+\to\N^+$ is a total recursive function,
from the definition of $S(m)$
it is easy to see that $\mathcal{T}$ is r.e.
Therefore it follows from \eqref{pr:SpeedLimitTheorem-RV:eq8} that
$\mathcal{T}$ is Martin-L\"of $P$-test.
Using \eqref{pr:SpeedLimitTheorem-RV:eq4} and \eqref{pr:SpeedLimitTheorem-RV:eq7}
we have that
\begin{equation}\label{pr:SpeedLimitTheorem-RV:eq9}
  \alpha\in\osg{\mathcal{T}_m}
\end{equation}
for all $m\in\N^+$,
and therefore $\alpha$ is not Martin-L\"of $P$-random, as desired.

Furthermore, suppose that $X(a)$ is a rational for all $a\in\Omega$.
Then, since $\mu$ is a computable real and
the function $f\colon\N^+\to\N^+$ is a total recursive function,
it is easy to see that one can effectively enumerate all the elements of the finite set $S(m)$, given $m\in\N^+$.
This is obvious in the case of $\mu\in\Q$.
This also holds true in the case of $\mu\notin\Q$ because, in such a case, either
\[
  \abs{\sum_{i=1}^{4^k} X(\tau(i))-4^k \mu}< 2^{k+1}
  \quad\text{ or }\quad
  \abs{\sum_{i=1}^{4^k} X(\tau(i))-4^k \mu}> 2^{k+1}
\]
holds for each $k\in\N^+$ and $\tau\in\Omega^{4^k}$,
and moreover $\mu$ is computable.
Thus, since $P$ is
a computable finite probability space,
it follows that
$\Bm{P}{\osg{S(m)}}$, i.e., $\Bm{P}{\osg{\mathcal{T}_m}}$, is uniformly computable in $m$,
and thus $\mathcal{T}$ is Schnorr $P$-test.
Hence, it follows from \eqref{pr:SpeedLimitTheorem-RV:eq9} that
$\alpha$ is not Schnorr  $P$-random.
This completes the proof.
\end{proof}

Theorems~\ref{EffectiveLLN-RV} and \ref{SpeedLimitTheorem-RV}
together
lead to
the following two theorems,
Theorem~\ref{thm:main-ML} and Theorem~\ref{thm:main-Schnorr}
below.

\begin{theorem}
[Main result II regarding algorithmic randomness]
\label{thm:main-ML}
Let $\Omega$ be an alphabet, and let $P\in\PS(\Omega)$.
Suppose that $P$ is computable.
Let $X$ be a real random variable on $\Omega$.
Suppose that
$X(a)$ is a computable real for all $a\in\Omega$
and $V(X)>0$.
Let $\alpha\in\Omega^\infty$. Suppose that $\alpha$ is Martin-L\"of $P$-random.
Then, for every real $t>0$, the following conditions~(i) and (ii) are equivalent to each other:
\begin{enumerate}
  \item There exists $M\in\N^+$ such that for every $n\ge M$ and every $k\in\N^+$
    if $k\ge n^t$ then
    \begin{equation*}%
      \abs{\frac{1}{k}\sum_{i=1}^k X(\alpha(i))-E(X)}<\frac{1}{n}.
    \end{equation*}
  \item $t>2$.
\end{enumerate}
\end{theorem}

\begin{proof}
Theorem~\ref{thm:main-ML} follows immediately from
Theorems~\ref{EffectiveLLN-RV} and \ref{SpeedLimitTheorem-RV}, and the fact that,
for every $Q\in\PS(\Omega)$ and $\beta\in\Omega^\infty$,
if $\beta$ is Martin-L\"of $Q$-random then $\beta$ is Schnorr $Q$-random.
\end{proof}

\begin{theorem}
[Main result III regarding algorithmic randomness]
\label{thm:main-Schnorr}
Let $\Omega$ be an alphabet, and let $P\in\PS(\Omega)$.
Suppose that $P$ is computable.
Let $X$ be a real random variable on $\Omega$.
Suppose that
$X(a)$ is a rational for all $a\in\Omega$
and $V(X)>0$.
Let $\alpha\in\Omega^\infty$. Suppose that $\alpha$ is Schnorr $P$-random.
Then, for every real $t>0$, the following conditions~(i) and (ii) are equivalent to each other:
\begin{enumerate}
  \item There exists $M\in\N^+$ such that for every $n\ge M$ and every $k\in\N^+$
    if $k\ge n^t$ then
    \begin{equation*}%
      \abs{\frac{1}{k}\sum_{i=1}^k X(\alpha(i))-E(X)}<\frac{1}{n}.
    \end{equation*}
  \item $t>2$.
\end{enumerate}
\end{theorem}

\begin{proof}
Theorem~\ref{thm:main-Schnorr} follows immediately from
Theorems~\ref{EffectiveLLN-RV} and \ref{SpeedLimitTheorem-RV}.
\end{proof}

\section{Effectivization of almost sure convergence in the strong law of large numbers,
and its absolute speed limit of convergence}
\label{Section:EffectiveSLLN}

In this section, we investigate
an effectivization of almost sure convergence in the strong law of large numbers,
and its absolute speed limit of convergence,
within the framework of probability theory.
This section does not depend on any results of the previous sections.
Thus,
this section can be read independently of the previous sections.

In this section, we consider
a general probability space $(\Omega,\mathcal{F},P)$
where $\Omega$ is the sample space, $\mathcal{F}$ is a $\sigma$-field in $\Omega$,
and $P$ is a probability measure on
$(\Omega,\mathcal{F})$.
Note that, in this section, the sample space $\Omega$ is not necessarily a finite set
like in the preceding sections.
Let $(\Omega,\mathcal{F},P)$ be a probability space, and
Let $X$ be a real random variable on $(\Omega,\mathcal{F},P)$.
In this section, we denote the \emph{mean} and \emph{variance} of $X$ by
$E[X]$ and $V[X]$, respectively. Thus, $V[X]=E[(X-E(X))^2]$.
See Billingsley~\cite{B95}, Chung~\cite{Chung01},
Durrett~\cite{Durrett19}, and
Klenke~\cite{Klenke20} for probability theory in general.

Frist, recall that the strong law of large numbers has the following form
\cite{Et81,B95,Chung01,Durrett19,Klenke20}.

\begin{theorem}%
[The strong law of large numbers]
Let $X_1,X_2,\dotsc$ be
independent and identically distributed real random variables on a probability space~$(\Omega,\mathcal{F},P)$.
Suppose that $E[\abs{X_1}]<\infty$.
Then the property
\begin{equation*}
  \lim_{n\to\infty}\frac{1}{n}\sum_{i=1}^n X_i=E[X_1]
\end{equation*}
holds almost surely, i.e.,
\begin{equation*}
  P\left(\left\{\omega\in\Omega\;\middle\vert\>\lim_{n\to\infty}\frac{1}{n}\sum_{i=1}^n X_i(\omega)=E[X_1]\right\}\right)=1.
\end{equation*}
\qed
\end{theorem}

In this section, we investigate an \emph{effectivization} of
almost sure convergence
in the strong law of large numbers above.
We will prove several theorems corresponding to ones
proved
in the preceding sections
regarding algorithmic randomness.
Note that any probability space
or any random variables
on it
which we consider in this section 
are not required to be computable at all in any sense.

First, we show the following theorem
regarding
probability theory,
which corresponds to Theorem~\ref{EffectiveLLN-RV}
regarding
algorithmic randomness
in Section~\ref{Section:EffectiveLLN=RV}.

\begin{theorem}
[Effectivization of almost sure convergence in the strong law of large numbers~I]
\label{th:SLLN-effective-convergence}
Let $X_1,X_2,\dotsc$ be
independent and identically distributed random variables on a probability space~$(\Omega,\mathcal{F},P)$.
Suppose that
there exist reals $a$ and $b$ with $a<b$
such that $a\le X_1\le b$ holds almost surely.
Then the following property holds almost surely:
\begin{quote}
For every real $t>2$ there exists $M\in\N^+$ such that
for every $n\ge M$ and every $k\in\N^+$ if $k\ge n^t$ then
\begin{equation*}
  \abs{\frac{1}{k}\sum_{i=1}^k X_i-E[X_1]}<\frac{1}{n}.
\end{equation*}
\end{quote}
This property is certainly $\mathcal{F}$-measurable.
\qed
\end{theorem}

In order to prove Theorem~\ref{SpeedLimitTheorem-SLLN},
we use Hoeffding's inequality below \cite{Hoeff63,BLM13}.

\begin{theorem}[Hoeffding's inequality, Hoeffding~\cite{Hoeff63}]\label{HI}
Let $(\Omega,\mathcal{F},P)$ be a probability space.
Let $a$ and $b$ be reals with $a<b$.
Let $X_1,X_2,\dotsc$ be independent and identically distributed
real random variables on $(\Omega,\mathcal{F},P)$
such that $a\le X_1\le b$ holds almost surely.
Then, for every real $\varepsilon>0$ and $n\in\N^+$, it holds that
\begin{equation*}%
  P\left(\abs{\frac{1}{n}\sum_{k=1}^n X_k-E[X_1]}\ge\varepsilon\right)
  \le 2\exp\left(-\frac{2\varepsilon^2}{(b-a)^2}n\right).
\end{equation*}
\qed
\end{theorem}

In order to prove Theorem~\ref{SpeedLimitTheorem-SLLN},
we also need the following lemma:

\begin{lemma}\label{EffectiveSLLN-lemma}
Let $a$ and $b$ be reals with $a<b$, and let $\epsilon$ be a positive real.
Then the double infinite sum
\begin{equation}\label{EffectiveSLLN-DIS}
  \sum_{n=1}^\infty\sum_{k=f_{\epsilon}(n)}^\infty
  \exp\left(-\frac{2k}{(b-a)^2n^2}\right)
\end{equation}
exists as a finite real,
where $f_{\epsilon}$ denotes a function  $f_{\epsilon}\colon\N^+\to\N^+$ defined by
$f_{\epsilon}(n)=\left\lceil n^{2+\epsilon} \right\rceil$.
\end{lemma}

\begin{proof}
In what follows, we denote $(b-a)^2/2$ by $c$.
Note that $c>0$ by the assumption of the lemma.
First,
for each $n\in\N^+$,
using the mean value theorem, we have that
\begin{equation}\label{EffectiveSLLN-MVT}
  1-\exp\left(-\frac{1}{cn^2}\right)
  >\exp\left(-\frac{1}{cn^2}\right)\frac{1}{cn^2}
  \ge \exp(-1/c)\frac{1}{cn^2}.
\end{equation}
Thus, for each $n\in\N^+$, we see that
\begin{equation}\label{EffectiveSLLN-ineq01}
\begin{split}
  \sum_{k=f_{\epsilon}(n)}^\infty \exp\left(-\frac{k}{cn^2}\right)
  &\le \sum_{l=0}^\infty\exp\left(-\frac{n^{\epsilon}}{c}\right)
  \exp\left(-\frac{l}{cn^2}\right)
  =\frac{\exp(-n^{\epsilon}/c)}{1-\exp(-1/(cn^2))} \\
  &<c\exp(1/c) n^2\exp(-n^{\epsilon}/c)),
\end{split}
\end{equation}
where the last inequality follows from the inequality~\eqref{EffectiveSLLN-MVT}.

We will reduce the convergence of the double infinite sum~\eqref{EffectiveSLLN-DIS}
to the existence of the incomplete gamma function $\Gamma(x,y)$ defined by
\begin{equation*}%
  \Gamma(x,y):=\int_{y}^\infty t^{x-1} e^{-t}dt,
\end{equation*}
where $x$ and $y$ are arbitrary reals satisfying that $x>0$ and $y\ge 0$.
It is easy to see that
the improper integral in the definition above exists certainly for such reals $x$ and $y$.
See for instance Jameson~\cite{Jam16}
for the detail of the properties of the incomplete gamma function~$\Gamma(x,y)$.

Now, we choose any specific $N\in\N^+$ with
$N^{\epsilon}\ge 2c/\epsilon$.
Applying the method of integration of substitution to
the improper integral $\Gamma(3/\epsilon,N^{\epsilon}/c)$ with $t=u^{\epsilon}/c$, we see that
the improper integral
\[
  \int_N^\infty u^2 \exp(-u^{\epsilon}/c) du
\]
exists and equals
\[
  \frac{c^{3/\epsilon}}{\epsilon}\Gamma(3/\epsilon,N^{\epsilon}/c).
\]
Thus, since $u^2\exp(-u^{\epsilon}/c)$ is a strictly decreasing function of $u$ for all
positive real
$u$ with $u^{\epsilon}\ge 2c/\epsilon$,
we have that
\begin{equation}\label{Lemma-ML-Gamma-eq01}
  \sum_{n=N+1}^\infty n^2\exp(-n^{\epsilon}/c)
  \le\int_N^\infty u^2 \exp(-u^{\epsilon}/c) du
  =\frac{c^{3/\epsilon}}{\epsilon}\Gamma(3/\epsilon,N^{\epsilon}/c).
\end{equation}
Hence, it follows from \eqref{EffectiveSLLN-ineq01} and \eqref{Lemma-ML-Gamma-eq01}
that the double inifinite sum~\eqref{EffectiveSLLN-DIS} converges to a finite real.
\end{proof}

Then, the proof of Theorem~\ref{th:SLLN-effective-convergence} is given as follows.

\begin{proof}[Proof of Theorem~\ref{th:SLLN-effective-convergence}]
Since there exist reals $a$ and $b$ with $a<b$
such that $a\le X_1\le b$ holds almost surely,
it follows from Hoeffding's inequality, Theorem~\ref{HI}, that
\begin{equation}\label{Effective-SLLN-HI:eq1}
  P\left(\abs{\frac{1}{k}\sum_{i=1}^k X_i-\mu}\ge\frac{1}{n}\right)
  \le 2\exp\left(-\frac{2k}{(b-a)^2n^2}\right)
\end{equation}
for every $n\in\N^+$ and every $k\in\N^+$,
where $\mu:=E[X_1]$.

Now, let $\epsilon$ be an arbitrary positive real.
We then define a function $f_{\epsilon}\colon\N^+\to\N^+$ by
\[
  f_{\epsilon}(n):=\left\lceil n^{2+\varepsilon} \right\rceil.
\]
For each $m\in\N^+$, using \eqref{Effective-SLLN-HI:eq1} we have that
\begin{equation}\label{Effective-SLLN-HI:eq2} 
\begin{split}
  P\left(\exists\,n\ge m\;\,\exists\,k\ge f_{\epsilon}(n)
  \left[\>\abs{\frac{1}{k}\sum_{i=1}^k X_i-\mu}\ge\frac{1}{n}\>\right]\right)
  &\le\sum_{n=m}^\infty\sum_{k=f_{\epsilon}(n)}^\infty P\left(\abs{\frac{1}{k}\sum_{i=1}^k X_i-\mu}\ge\frac{1}{n}\right) \\
  &\le\sum_{n=m}^\infty\sum_{k=f_{\epsilon}(n)}^\infty 2\exp\left(-\frac{2k}{(b-a)^2n^2}\right),
\end{split}
\end{equation}
where the symbol~$\vee$ denotes the logical disjunction.
Since the most right-hand side of \eqref{Effective-SLLN-HI:eq2} converges to $0$
as $m\to\infty$ due to Lemma~\ref{EffectiveSLLN-lemma},
we have that
\begin{equation*}%
  P\left(\forall\,m\in\N^+\;\,\exists\,n\ge m\;\,\exists\,k\ge f_{\epsilon}(n)
  \left[\>\abs{\frac{1}{k}\sum_{i=1}^k X_i-\mu}\ge\frac{1}{n}\>\right]\right)=0.
\end{equation*}
Since $\epsilon$ is an arbitrary positive real and
any countable union of null sets is still a null set,
we have that
\begin{equation}\label{Effective-SLLN-HI:eq4}
  P\left(\exists\,\epsilon\in\Q^+\;\,\forall\,m\in\N^+\;\,\exists\,n\ge m\;\,\exists\,k\ge f_{\epsilon}(n)
  \left[\>\abs{\frac{1}{k}\sum_{i=1}^k X_i-\mu}\ge\frac{1}{n}\>\right]\right)=0,
\end{equation}
where $\Q^+$ denotes the set of positive rationals.
It is then easy to see that
the event in the left-hand side of
\eqref{Effective-SLLN-HI:eq4} is equivalent to
the event
\begin{equation*}%
  \exists\,\epsilon\in\R^+\;\,\forall\,m\in\N^+\;\,\exists\,n\ge m\;\,\exists\,k\ge f_{\epsilon}(n)
  \left[\>\abs{\frac{1}{k}\sum_{i=1}^k X_i-\mu}\ge\frac{1}{n}\>\right],
\end{equation*}
where $\R^+$ denotes the set of positive reals.
Hence, the result follows.
\end{proof}

The following corollary
corresponds to Corollary~\ref{cor:prf-EffectiveLLN-RV}
regarding
algorithmic randomness
in Section~\ref{Section:EffectiveLLN=RV}.

\begin{corollary}
[Effectivization of almost sure convergence in the strong law of large numbers~II]
\label{cor:SLLN-effective-convergence}
Let $X_1,X_2,\dotsc$ be
independent and identically distributed random variables on a probability space~$(\Omega,\mathcal{F},P)$.
Suppose that
there exist reals $a$ and $b$ with $a<b$
such that $a\le X_1\le b$ holds almost surely.
Then the following property holds almost surely:
\begin{quote}
There exists a primitive recursive function $f\colon\N\to\N$ such that
for every $n\in\N^+$ and every $k\in\N^+$ if $k\ge f(n)$ then
\begin{equation*}
  \abs{\frac{1}{k}\sum_{i=1}^k X_i-E[X_1]}<\frac{1}{n}.
\end{equation*}
\end{quote}
This property is certainly $\mathcal{F}$-measurable.
\end{corollary}

\begin{proof}
Let $\omega\in\Omega$.
Suppose that
for every real $t>2$ there exists $M\in\N^+$ such that
for every $n\ge M$ and every $k\in\N^+$ if $k\ge n^t$ then
\begin{equation}\label{cor:SLLN-effective-convergence:eq1}
  \abs{\frac{1}{k}\sum_{i=1}^k X_i(\omega)-E[X_1]}<\frac{1}{n}.
\end{equation}
Then,
by choosing $t$ to be $3$
in particular,
we have that there exists $M\in\N^+$ such that
for every $n\ge M$ and every $k\in\N^+$ if $k\ge n^3$ then
the inequality~\eqref{cor:SLLN-effective-convergence:eq1} holds.
We define a function $f\colon\N\to\N$ by the condition that
$f(n):=n^3$ if $n\ge M$ and $f(n):=M^3$ otherwise.
Then, on the one hand, it follows that $f\colon\N\to\N$ is a primitive recursive function.
On the other hand,
we
have
that for every $n\in\N^+$ and every $k\in\N^+$ if $k\ge f(n)$ then
the inequality~\eqref{cor:SLLN-effective-convergence:eq1} holds.

Hence, it follows from Theorem~\ref{th:SLLN-effective-convergence} that
the event
\begin{equation}\label{cor:SLLN-effective-convergence:eq2}
  \bigcup_{\text{$f$: p.~r.}}\,
  \bigcap_{n=1}^\infty\bigcap_{k=f(n)}^\infty
  \left\{\omega\in\Omega\,\middle\vert\,
    \abs{\frac{1}{k}\sum_{i=1}^k X_i(\omega)-E[X_1]}<\frac{1}{n}
  \right\}
\end{equation}
has the probability $1$,
where the left-most union is over all primitive recursive functions $f$ from $\N$ to $\N$.
Note that
this set~\eqref{cor:SLLN-effective-convergence:eq2}
is certainly $\mathcal{F}$-measurable,
since there are only countably infinitely many primitive recursive functions.
\end{proof}

The following theorem
in probability theory
corresponds to Theorem~\ref{SpeedLimitTheorem-RV}
regarding
algorithmic randomness
in Section~\ref{Section:EffectiveLLN=RV}.

\begin{theorem}
[Almost sure convergence speed limit theorem]
\label{SpeedLimitTheorem-SLLN}
Let $X_1,X_2,\dotsc$ be
independent and identically distributed real random variables on a probability space~$(\Omega,\mathcal{F},P)$.
Suppose that $E[(X_1)^2]<\infty$ and $V[X_1]>0$.
Then the following property holds almost surely:
\begin{quote}
For every real $t>0$, if there exists $M\in\N^+$ such that
for every $n\ge M$ and every $k\in\N^+$ if $k\ge n^t$ then
\begin{equation*}
  \abs{\frac{1}{k}\sum_{i=1}^k X_i-E[X_1]}<\frac{1}{n},
\end{equation*}
then $t>2$.
\end{quote}
This property is certainly $\mathcal{F}$-measurable.
\qed
\end{theorem}

In order to prove Theorem~\ref{SpeedLimitTheorem-SLLN},
we use the central limit theorem below
\cite{B95,Chung01,Durrett19,Klenke20}.

\begin{theorem}%
[The cetnral limit theorem]
\label{CLT}
Let $(\Omega,\mathcal{F},P)$ be a probability
space.
Let $X_1,X_2,\dotsc$ be independent and identically distributed
real random variables on $(\Omega,\mathcal{F},P)$
with $E[(X_1)^2]<\infty$.
Suppose that
$V[X_1]>0$.
Then for every reals $a$ and $b$ with $a<b$ it holds that
\begin{equation*}%
  \lim_{n\to\infty}P\left(a\le\frac{1}{\sqrt{nv}}\sum_{k=1}^n(X_k-\mu)\le b\right)
  =\int_{a}^{b} \frac{1}{\sqrt{2\pi}}e^{-x^2/2}dx,
\end{equation*}
where $\mu:=E[X_1]$ and $v:=V[X_1]$.
\qed
\end{theorem}

Then, the proof of Theorem~\ref{SpeedLimitTheorem-SLLN} is given as follows.

\begin{proof}[Proof of Theorem~\ref{SpeedLimitTheorem-SLLN}]
For each $n\in \N^+$, let $S_n:=X_1+\dots +X_n$.
Then, since $E[(X_1)^2]<\infty$ and $V[X_1]>0$,
it follows from the central limit theorem, Theorem~\ref{CLT}, that
for every reals $l_1$ and $l_2$ with $l_1<l_2$ it holds that
\begin{equation*}%
  \lim_{n\to\infty} P\left(l_1\sqrt{nv}\le S_n-n\mu\le l_2\sqrt{nv}\right)
  =\int_{l_1}^{l_2} \frac{1}{\sqrt{2\pi}}e^{-x^2/2}dx,
\end{equation*}
where $\mu:=E[X_1]$ and $v:=V[X_1]>0$.
It follows that
\begin{equation*}%
  \lim_{n\to\infty}
  P\left(-3\cdot 2^n\le S_{3\cdot 4^n}-3\cdot 4^n \mu\le 3\cdot 2^n\right)
  =\int_{-\sqrt{3/v}}^{\sqrt{3/v}}\frac{1}{\sqrt{2\pi}}e^{-x^2/2}dx<1.
\end{equation*}
Therefore there exist a positive integer $n_0$ and a real $r\in(0,1)$ such that
\begin{equation}\label{pr:SpeedLimitTheorem-SLLN:ieq1}
  P \left(\abs{S_{3\cdot 4^n}-3\cdot 4^n \mu}\le 3\cdot 2^n\right)<r
\end{equation}
for all $n\ge n_0$.

Now, let $m$ be an arbitrary integer with $m\ge n_0$.
First, note that, for each $n\in\N^+$ and $\omega\in\Omega$,
if $\abs{S_{4^n}(\omega)-4^{n}\mu}\le 2^{n}$ and
$\abs{S_{4^{n+1}}(\omega)-4^{n+1}\mu}\le 2^{n+1}$ then
\[
  \abs{(X_{4^n+1}+\dots +X_{4^{n+1}})(\omega)-3\cdot 4^{n}\mu}
  =\abs{(S_{4^{n+1}}(\omega)-S_{4^n}(\omega))-3\cdot 4^{n}\mu}
  \le 3\cdot 2^{n}.
\]
Therefore, for each $n\ge m$, we have that
\begin{equation}\label{pr:SpeedLimitTheorem-SLLN:eqs2}
\begin{split}
  &P\left(\bigwedge_{k=m}^{n+1}\abs{S_{4^k}-4^k\mu}\le 2^{k}\right) \\
  &\le P\left(\bigwedge_{k=m}^{n}\abs{S_{4^k}-4^k\mu}\le 2^{k}\;\&\;
    \abs{X_{4^n+1}+\dots +X_{4^{n+1}}-3\cdot 4^{n}\mu}\le 3\cdot 2^{n}\right) \\
  &=P\left(\bigwedge_{k=m}^{n}\abs{S_{4^k}-4^k\mu}\le 2^{k}\right)
    P\left(\abs{X_{4^n+1}+\dots +X_{4^{n+1}}-3\cdot 4^{n}\mu}\le 3\cdot 2^{n}\right) \\
  &=P\left(\bigwedge_{k=m}^{n}\abs{S_{4^k}-4^k\mu}\le 2^{k}\right)
    P\left(\abs{S_{3\cdot 4^n}-3\cdot 4^n\mu}\le 3\cdot 2^n\right) \\
  &\le P\left(\bigwedge_{k=m}^{n}\abs{S_{4^k}-4^k\mu}\le 2^{k}\right)r,
\end{split}
\end{equation}
where the symbol~$\wedge$, as well as $\&$, denotes the logical conjunction, and
the first and second inequalities follow from the
assumption
that
$X_1,X_2,\dotsc$ are independent and identically distributed random variables
on $(\Omega,\mathcal{F},P)$ and
the last  inequality follows from \eqref{pr:SpeedLimitTheorem-SLLN:ieq1}.
Thus, for each $n\ge m$, it follows from \eqref{pr:SpeedLimitTheorem-SLLN:eqs2} that
\begin{equation*}%
  P\left(\bigwedge_{k=m}^{n}\abs{S_{4^k}-4^k\mu}\le 2^{k}\right)
  \le P\left(\bigwedge_{k=m}^{m}\abs{S_{4^k}-4^k\mu}\le 2^{k}\right)r^{n-m} \le r^{n-m}.
\end{equation*}
Since $0<r<1$, it follows that
\begin{equation*}%
  P\left(\forall\,k\ge m\left[\>\abs{S_{4^k}-4^k\mu}\le 2^{k}\>\right]\right)=0.
\end{equation*}
Thus, since $m$ is an arbitrary integer with $m\ge n_0$, we have that
\begin{equation}\label{pr:SpeedLimitTheorem-SSLN:eq7}
  P\left(\exists\,m\ge n_0\;\,\forall\,k\ge m\left[\>\abs{S_{4^k}-4^k\mu}\le 2^{k}\>\right]\right)=0.
\end{equation}

Let $\omega\in\Omega$.
Suppose that
there exist a real $t$ with $0<t\le 2$ and $M\in\N^+$ such that
for every $n\ge M$ and every $k\in\N^+$ if $k\ge n^t$ then
\begin{equation}\label{pr:SpeedLimitTheorem-SSLN:eq8}
  \abs{\frac{1}{k}\sum_{i=1}^k X_i(\omega)-\mu}<\frac{1}{n}.
\end{equation}
Then,
we have that
there exists an integer $m_0\ge n_0$ such that for every $k\ge m_0$ it holds that
$\abs{S_{4^k}(\omega)-4^k\mu}\le 2^{k}$.
Note here that
the condition~\eqref{pr:SpeedLimitTheorem-SSLN:eq8} on $\omega$
is equivalent to
the condition on $\omega$
that there exist a \emph{rational} $t$ with $0<t\le 2$ and $M\in\N^+$ such that
for every $n\ge M$ and every $k\in\N^+$ if $k\ge n^t$ then
\begin{equation*}%
  \abs{\frac{1}{k}\sum_{i=1}^k X_i(\omega)-\mu}<\frac{1}{n}.
\end{equation*}
Therefore,
the condition~\eqref{pr:SpeedLimitTheorem-SSLN:eq8} on $\omega$
is $\mathcal{F}$-measurable.

Hence, it follows from \eqref{pr:SpeedLimitTheorem-SSLN:eq7} that
the following property is $\mathcal{F}$-measurable, and holds almost surely:
For every real $t>0$ if there exists $M\in\N^+$ such that
for every $n\ge M$ and every $k\in\N^+$ if $k\ge n^t$ then
\begin{equation*}
  \abs{\frac{1}{k}\sum_{i=1}^k X_i-\mu}<\frac{1}{n},
\end{equation*}
then $t>2$.
\end{proof}

The following theorem
regarding probability theory
corresponds to
Theorems~\ref{thm:main-ML} and \ref{thm:main-Schnorr}
regarding algorithmic randomness
in Section~\ref{Section:EffectiveLLN=RV}.

\begin{theorem}
[Main result regarding probability theory]
\label{thm:main-SLLN}
Let $X_1,X_2,\dotsc$ be
independent and identically distributed random variables on a probability space~$(\Omega,\mathcal{F},P)$.
Suppose that $V[X_1]>0$ and there exist reals $a$ and $b$ with $a<b$
such that $a\le X_1\le b$ holds almost surely.
Then the following property holds almost surely:
\begin{quote}
For every real $t>0$, the following conditions~(i) and (ii) are equivalent to each other:
\begin{enumerate}
\item There exists $M\in\N^+$ such that
for every $n\ge M$ and every $k\in\N^+$ if $k\ge n^t$ then
\begin{equation*}
  \abs{\frac{1}{k}\sum_{i=1}^k X_i-E[X_1]}<\frac{1}{n}.
\end{equation*}
\item $t>2$.
\end{enumerate}
\end{quote}
This property is certainly $\mathcal{F}$-measurable.
\end{theorem}

\begin{proof}
Since there exist reals $a$ and $b$ with $a<b$ such that $a\le X_1\le b$ holds almost surely,
we have that $E[(X_1)^2]<\infty$.
Thus, Theorem~\ref{thm:main-SLLN} follows immediately from
Theorems~\ref{th:SLLN-effective-convergence} and \ref{SpeedLimitTheorem-SLLN}.
\end{proof}

\section*{Acknowledgments}

This work was supported by JSPS KAKENHI Grant Number 22K03409.

\end{document}